\def\draftdate{May 5, 2015}
\date{\draftdate}

\documentclass{amsart}

\usepackage{amssymb}
\usepackage{xy}

\DeclareMathAlphabet{\mathscr}{U}{rsfs}{m}{n}

\newcommand{\B}{j}

\newcommand{\G}{\mathbb{T}}
\newcommand{\an}{a\ }
\newcommand{\cyc}{t}

\newcommand{\dPhi}[1]{\bL\Phi^{#1}}

\let\iso\cong
\let\sma\wedge

\renewcommand{\to}{\mathchoice{\longrightarrow}{\rightarrow}{\rightarrow}{\rightarrow}}
\newcommand{\from}{\mathchoice{\longleftarrow}{\leftarrow}{\leftarrow}{\leftarrow}}

\newcommand{\overto}[1]{\xrightarrow{\,#1\,}}

\newcommand{\phat}{^{\scriptscriptstyle\wedge}_{p}}

\let\catsymbfont\mathcal

\newcommand{\aD}{{\catsymbfont{D}}}
\newcommand{\aF}{{\catsymbfont{F}}}
\newcommand{\aI}{{\catsymbfont{I}}}

\newcommand{\aU}{{\catsymbfont{U}}}
\newcommand{\aT}{{\catsymbfont{T}}}

\newcommand{\bC}{{\mathbb{C}}}
\newcommand{\pC}{{\mathbb{C}_{p}}}
\newcommand{\bI}{{\mathbb{I}}}
\newcommand{\bJ}{{\mathbb{J}}}
\newcommand{\bL}{{\mathbb{L}}}
\newcommand{\bN}{{\mathbb{N}}}
\newcommand{\bR}{{\mathbb{R}}}
\newcommand{\bZ}{{\mathbb{Z}}}

\newcommand{\sI}{\mathscr{I}}
\newcommand{\sJ}{\mathscr{J}}

\def\quickop#1{\expandafter\DeclareMathOperator\csname
#1\endcsname{#1}}
\quickop{id}\quickop{Id}
\quickop{holim}
\quickop{hocolim}
\quickop{Fix}
\quickop{fin}
\quickop{Mic}
\quickop{Eq}\quickop{hoEq}
\let\equalizer=\Eq
\let\hoequalizer=\hoEq

\DeclareMathOperator*\lcolim{colim}
\DeclareMathOperator*\lholim{holim}

\newtheorem{thm}{Theorem}
\numberwithin{thm}{section}

\newtheorem{cor}[thm]{Corollary}
\newtheorem{lem}[thm]{Lemma}
\newtheorem{prop}[thm]{Proposition}
\theoremstyle{definition}
\newtheorem{defn}[thm]{Definition}
\newtheorem{cons}[thm]{Construction}

\theoremstyle{remark}
\newtheorem{rem}[thm]{Remark}
\newtheorem{example}[thm]{Example}

\makeatletter\let\c@equation\c@thm\makeatother

\ifx\texorpdfstring\undefined\def\texorpdfstring#1#2{#1}\fi

\xyoption{arrow}
\xyoption{curve}
\xyoption{matrix}
\xyoption{cmtip}
\xyoption{frame}
\SelectTips{cm}{}

\newcommand{\term}[1]{\textit{#1}}

\newdir{ >}{{}*!/-5pt/\dir{>}}

\begin{document}

\title[Cyclotomic spectra]%
{The homotopy theory of cyclotomic spectra\expandafter\footnotemark}

\author{Andrew J. Blumberg}
\address{Department of Mathematics, The University of Texas,
Austin, TX \ 78712}
\email{blumberg@math.utexas.edu}
\author{Michael A. Mandell}
\address{Department of Mathematics, Indiana University,
Bloomington, IN \ 47405}
\email{mmandell@indiana.edu}


\begin{abstract}
We describe spectral model category structures on the categories of
cyclotomic spectra and $p$-cyclotomic spectra (in orthogonal spectra)
with triangulated homotopy categories.  We show that the functors $TR$
and $TC$ are corepresentable in these categories.  Specifically, the
derived mapping spectrum out of the sphere spectrum in the category of
cyclotomic spectra corepresents the finite completion of $TC$ and the 
derived mapping spectrum out of the sphere spectrum in the category of
$p$-cyclotomic spectra corepresents the $p$-completion of $TC(-;p)$.
\end{abstract}

\maketitle

\footnotetext{This is a last draft version (\draftdate) of an article whose final
version was first published in
\textit{Geometry \& Topology} in 2015, published by Mathematical
Sciences Publishers.\par 
\copyright\ 2015 Andrew J. Blumberg and Michael A. Mandell}

\section{Introduction}

Topological cyclic homology ($TC$) has proved to be an enormously
successful tool for studying algebraic $K$-theory and $K$-theoretic
phenomena.  After finite completion, relative $K$-theory for certain
pairs is equivalent to relative $TC$ via the cyclotomic trace map and
$TC$ can be computed using the methods of equivariant stable homotopy
theory.

The construction of $TC$ begins with a cyclotomic spectrum; this is
\an $\G$-equivariant spectrum equipped with additional structure that
mimics the structure seen on the free suspension spectrum of the free
loop space, $\Sigma^{\infty}_{+} \Lambda X$.  Here $\G$ denotes the
circle group of unit complex numbers and $\Lambda X$ the $\G$-space of
maps from $\G$ to $X$.  The $n$-th
root map induces an isomorphism $\rho_{n} \colon \G \iso \G / C_{n}$,
and this induces an isomorphism of $\G$-spaces $\rho_{n}^{*}(\Lambda
X)^{C_{n}} \iso \Lambda X$, where $C_n$ is the cyclic subgroup of
order $n$ and $\rho_{n}^{*}$ is the change of
group functor along $\rho_{n}$.  This isomorphism then gives rise to
an equivalence of $\G$-spectra $\rho_{n}^{*}\Phi^{C_{n}}
\Sigma^{\infty}_{+} \Lambda X \simeq \Sigma^{\infty}_{+} \Lambda X$,
where $\Phi^{C_{n}}$ denotes the derived geometric fixed point functor, and
$\rho_{n}^{*}$ denotes both change of groups and change of universe.

Although $TC$ has been around for over twenty years, there has been
relatively little investigation of the nature of cyclotomic spectra or
the homotopy theory associated to the category of cyclotomic spectra.
Recently, in the course of proving the degeneration of the
noncommutative ``Hodge to de Rham'' spectral sequence, Kaledin has
described the close connection between cyclotomic spectra
and Dieudonn\'e modules and the relationship 
between TC and syntomic cohomology.  This work led him to make
conjectures~\cite[\S 7]{Kaledin-ICM2010} regarding the structure of
the category of cyclotomic spectra and its relationship to $TC$.  The
purpose of this paper is to prove these conjectures.

After a review of background, we begin in Section~\ref{sec:def} by
setting up a point-set category of cyclotomic spectra as a category of
orthogonal $\G$-spectra \cite{MM} with extra structure; see
Definition~\ref{defn:cyc}.  Topological Hochschild homology ($THH$)
provides the primary source of examples of cyclotomic spectra.  We
also set up a (significantly simpler) point-set category of
$p$-cyclotomic spectra: A $p$-cyclotomic spectrum is an orthogonal
$\G$-spectrum $X$ together with a map of $\G$-spectra
\[
\cyc \colon \rho_{p}^{*}\Phi^{C_{p}}T\to T
\]
from the (point-set) geometric fixed points of $T$ back to $T$ such
that the composite in the homotopy category from the derived geometric
fixed points is an $\aF_{p}$-equivalence, i.e., induces an isomorphism
on homotopy groups $\pi^{C_{p^{n}}}_{*}$ for all $n\geq 0$
(Definition~\ref{defn:pcyc}).  Since in most examples one works with
the $p$-cyclotomic structure, in the remainder of this introduction we
focus on this case for expositional simplicity.

Since the definition of cyclotomic spectra includes a
homotopy-theoretic constraint, it is unreasonable to expect any
category of cyclotomic spectra to be closed under general (or even
finite) limits or colimits.  Thus, we cannot expect a model category
of cyclotomic spectra; nevertheless, we show that our category of
cyclotomic spectra admits a \textit{model structure} in the sense of
\cite[1.1.3]{Hovey-ModelCat}: It has subcategories of cofibrations,
fibrations, and weak equivalences that satisfy Quillen's closed model
category axioms.  Moreover, the category of cyclotomic spectra admits
finite coproducts and products, pushouts over cofibrations, and
pullbacks over fibrations.  These limits and colimits suffice to
construct the entirety of the homotopy theory set up in Chapter~I of
\cite{Quillen-HA}, and much of the abstract homotopy theory developed
since; e.g., see~\cite{ABC} where this is worked out in even
greater generality.  As an example, since the category of cyclotomic
spectra additionally admits filtered colimits along cofibrations, we
can deduce that it has all homotopy colimits.

\begin{defn}\label{defn:modelstar}
A \term{model* category} is a category with a model structure
\cite[1.1.3]{Hovey-ModelCat} and which admits finite coproducts and
products, pushouts over cofibrations, and pullbacks over fibrations. 
\end{defn}

One big advantage of model* categories over model categories is that
any subcategory of a model* category that is closed under weak
equivalences, finite products and coproducts, pushouts over
cofibrations, and pullbacks over fibrations is a model* category with
the inherited model structure.  We take advantage of this in the proof
of the following theorem, which is our main theorem on the homotopy
theory of $p$-cyclotomic spectra.

\begin{thm}\label{thm:modelstar}
The category of $p$-cyclotomic spectra is a model* category with weak
equivalences the weak equivalences of the underlying non-equivariant
orthogonal spectra and with fibrations the $\aF_{p}$-fibrations (see
\cite[IV.6.5]{MM} or Theorem~\ref{thm:osflocalmod} below) of the
underlying orthogonal $\G$-spectra, where $\aF_{p}=\{C_{p^{n}}\}$.
\end{thm}

The model* category of $p$-cyclotomic spectra has additional
structure.  Clearly, it inherits an enrichment over spaces from the
category of orthogonal $\G$-spectra.  We show in
Section~\ref{sec:def} that it in fact inherits an enrichment over
non-equivariant orthogonal spectra.  This enrichment is compatible
with the model structure in the sense that the analogue of Quillen's
SM7 axiom holds; see Theorem~\ref{thm:SM7}.  In particular, the
homotopy category of $p$-cyclotomic spectra becomes triangulated with
the usual definition of distinguished triangles, and we
have a good construction of intrinsic mapping spectra.  

\begin{thm}\label{thm:stable}
The model structure on $p$-cyclotomic spectra has an enrichment over orthogonal
spectra.  The homotopy category of $p$-cyclotomic spectra is
triangulated with the shift functor given by suspension and the
distinguished triangles determined by the 
cofiber sequences specified by the model structure
(q.v.~\cite[6.2.6]{Hovey-ModelCat}). 
\end{thm}

In fact, the homotopy type of the mapping spectra turns out to have a relatively
straightforward description in terms of the underlying orthogonal
spectra and the structure map $\cyc$.  See Theorem~\ref{thm:newFp} for
a precise statement.

Finally, $TC(-;p)$ has an intrinsic interpretation in the context of
the homotopy category of $p$-cyclotomic spectra (after $p$-completion).  The
sphere spectrum $S$ has a canonical cyclotomic structure using the
canonical identification of the geometric fixed points of $S$ as $S$ (see
Example~\ref{exa:sphere}).  We can identify the $p$-completion of the
right derived functor of $TC(-;p)$ as the derived mapping spectrum out
of $S$.

\begin{thm}\label{thm:main}
Let $T$ be a $p$-cyclotomic spectrum.  Then the derived mapping
spectrum from the sphere spectrum $S$ to $T$ in the homotopy category
of $p$-cyclotomic spectra becomes naturally isomorphic to the
right derived functor of $TC(T;p)$ after $p$-completion.
Moreover, the natural isomorphism of $p$-completed right derived 
functors is canonical. 
\end{thm}

This confirms the conjecture of Kaledin \cite[7.9]{Kaledin-ICM2010}.
Also, this theorem gives a motivic interpretation of $TC(-;p)$,
viewing the triangulated homotopy category of $p$-cyclotomic spectra
as a category of ``$p$-cyclotomic motives'' associated to
non-commutative schemes (viewed as spectral categories, with $THH$ as
the realization functor).

See Section~\ref{sec:corep} for additional corepresentability results.

\bigskip

The question of developing a homotopy theory for cyclotomic spectra
was first asked by Ib Madsen in a problem session at the Stable
Homotopy Workshop at the Fields Institute in January 1996 and was
suggested to the authors by Lars Hesselholt.  The authors would like
to thank Vigleik Angeltveit for helpful comments.  The first author
thanks Haynes Miller and the MIT Math department for their hospitality.
The first author was supported in part by NSF grants
DMS-0906105, DMS-1151577.
The second author was supported in part by NSF grant
DMS-1105255

\section{Review of orthogonal \texorpdfstring{$\G$}{T}-spectra}

In this section, we give a brief review of the definition of
orthogonal $\G$-spectra and the geometric fixed point functors.  Along
the way we provide some new technical results that are needed in later
sections.  We begin with some preliminaries about the categories of
$\G$-spaces we work with.

We work throughout with the category $\aU$ of compactly generated weak
Hausdorff spaces, the objects of which we call \term{spaces}.  (As we
never use more general topological spaces, this will cause no
confusion.)  We use $\aT$ to denote the category of \term{based
spaces}, which is the undercategory in $\aU$ of the one-point space
$*=\{\{\}\}$.  The category $\aU$ is complete, cocomplete, and cartesian
closed.  The category $\aT$ is complete, cocomplete, and closed
symmetric monoidal under the smash product; we also regard $\aT$ as
enriched over $\aU$ by the forgetful functor, which is lax
symmetric monoidal.  The categories $\G \aU$ and
$\G \aT$ of \term{$\G$-spaces} and \term{based $\G$-spaces} are 
by definition the category of $\G$-objects in $\aU$ and $\aT$, respectively, where $\G$
denotes the circle group, the Lie group of unit complex numbers.
These categories are complete and cocomplete with the
limits and colimits constructed in $\aU$ and $\aT$ respectively.
The categories $\G\aU$ and $\G \aT$ are closed symmetric monoidal,
with product given by the cartesian product and smash product
respectively, and with function objects the function spaces from $\aU$
and $\aT$ endowed with the conjugation $\G$-action.
The $\G$-fixed point functor $\G\aU\to \aU$ and $\G\aT\to \aT$ is
symmetric monoidal and gives $\aU$ and $\aT$ enrichments,
respectively. 

There are several equivalent formulations of the 
category of orthogonal $\G$-spectra.  The simplest definition of
orthogonal $\G$-spectra~\cite[II.2]{MM} turns out to be less
technically convenient for our purposes than the reformulation
in~\cite[\S II.4]{MM} in terms of diagram spaces.  Recall that for a
skeletally small category $\aD$ enriched in $\G\aT$, a
\term{$\aD$-space} is a $\G\aT$-enriched functor from $\aD$ to $\G\aT$ and a
morphism of $\aD$-spaces is an enriched natural transformation of
enriched functors.  The category of $\aD$-spaces then has an
enrichment in $\G\aT$ given by the usual limit formula.  

For brevity, in what follows we write \term{orthogonal
$\G$-representation} to mean finite dimensional real inner product
space with $\G$-action by isometries (and not the isomorphism class of
such an object).  As in \cite[II.4.1]{MM}, for orthogonal
$\G$-representations $V$ and $W$, let $\sI(V,W)$ denote the $\G$-space
of (non-equivariant) linear isometries from $V$ to $W$.  Let $E(V,W)$
denote the subbundle of the product $\G$-bundle $\sI(V,W) \times W$
consisting of the points $(f,x)$ where $x$ is in the orthogonal
complement of $f(V)$.  Let $\sJ_{\G}(V,W)$ denote the Thom $\G$-space
of $E(V,W)$.  Composition of isometries and addition in the codomain
vector space induces composition maps
\[
\sJ_{\G}(W,Z) \sma \sJ_{\G}(V,W) \to \sJ_{\G}(V,Z),
\]
which together with the obvious identity elements
make $\sJ_{\G}$ a category enriched in based $\G$-spaces (with objects
the orthogonal $\G$-representations).  

\begin{defn}[{\cite[II.4.3]{MM}}]\label{defn:orthspec}
The category of \term{orthogonal $\G$-spectra} is the category of $\sJ_{\G}$-spaces.
\end{defn}

As discussed above, the category of orthogonal $\G$-spectra inherits an
enrichment in $\G\aT$.  In addition, the category of orthogonal
$\G$-spectra is a closed symmetric monoidal category under the smash
product constructed in~\cite[II.3]{MM} and in particular has internal
function objects: For $X$,$Y$ orthogonal $\G$-spectra, we let $F_{\G}(X,Y)$
denote the orthogonal $\G$-spectrum of maps (analogous to the
$\G$-space of maps between $\G$-spaces).  We denote its $\G$-fixed
point non-equivariant orthogonal spectrum as $F^{\G}(X,Y)$, which we
regard as the spectrum of $\G$-equivariant maps
from $X$ to $Y$.

We now turn to the discussion of the fixed point functors.  The
advantage of the diagram space definition (Definition~\ref{defn:orthspec})
over the spacewise definition of orthogonal $\G$-spectra is that the
diagram space definition makes it easier
to define the (point-set)
fixed point and geometric fixed point functors.  For $C\leq\G$ a closed
subgroup, consider $\sJ^{C}_{\G}(V,W)$, the $\G/C$-space of $C$-fixed
points of $\sJ_{\G}(V,W)$.  Identity elements and composition in
$\sJ_{\G}$ restrict appropriately, making $\sJ^{C}_{\G}$ a category
enriched in based $\G/C$-spaces.  Moreover, we have an evident
$\G/C$-enriched functor $\tilde{q}_C \colon \sJ_{\G/C}\to
\sJ^{C}_{\G}$ induced by regarding an orthogonal $\G/C$-representation
$V$ as an orthogonal $\G$-representation $q^{*}_{C}V$ via the quotient
map $q_{C} \colon \G\to \G/C$.  Given an orthogonal $\G$-spectrum $X$,
the fixed points $(X(V))^{C}$ form a $\sJ^{C}_{\G}$-space, i.e., a
based $\G/C$-space enriched functor from $\sJ^{C}_{\G}$ to based $\G/C$-spaces.
We can then compose with the enriched functor $\tilde{q}_C$ to obtain
an orthogonal $\G/C$-spectrum.

\begin{defn}[{\cite[V.3.1]{MM}}]\label{defn:catfix}
Let $X$ be an orthogonal $\G$-spectrum.  For $C\leq\G$ a closed subgroup,
let $X^{C}$ be the orthogonal $\G/C$-spectrum defined by $X^{C}(V)=
(X(q^{*}_{C}V))^{C}$, with $\sJ_{\G/C}$-space structure induced via
the enriched functor $\tilde{q}_C$ as above.  We call this functor
$(-)^{C}$ from orthogonal $\G$-spectra to orthogonal $\G/C$-spectra
the (point-set) \term{categorical fixed point functor}.
\end{defn}

We also have a based $\G/C$-space enriched functor $\phi \colon
\sJ^{C}_{\G}\to \sJ_{\G/C}$ which sends the orthogonal
$\G$-representation $V$ to the orthogonal $\G/C$-representation
$V^{C}$.  Enriched left Kan extension along $\phi$ constructs a
functor from $\sJ^{C}_{\G}$-spaces to $\sJ_{\G/C}$-spaces.  Applying
this to the fixed point $\sJ^{C}_{\G}$-space obtained from a
orthogonal $\G$-spectrum, we get the (point-set) geometric fixed point
functor.

\begin{defn}[{\cite[V.4.3]{MM}}]\label{defn:geofix}
Let $X$ be an orthogonal $\G$-spectrum.  For $C\leq\G$ a closed subgroup,
let $\Fix^{C}X$ be the $\sJ^{C}_{\G}$-space defined by
$\Fix^{C}X(V)=(X(V))^{C}$, and let $\Phi^{C}X$ be the orthogonal
$\G/C$-spectrum obtained from $\Fix^{C}X$ by enriched left Kan
extension along the enriched functor $\phi \colon
\sJ^{C}_{\G}\to \sJ_{\G/C}$. We call the functor $\Phi^{C}$ from
orthogonal $\G$-spectra to orthogonal $\G/C$-spectra the (point-set)
\term{geometric fixed point functor}.
\end{defn}

The categorical fixed point functor is a right adjoint \cite[V.3.4]{MM}
and so preserves all limits.  The geometric fixed point functor is
not a left adjoint but has the feel of a left adjoint, preserving all
the colimits preserved by the fixed point functor on based $\G$-spaces; this
includes pushouts over levelwise closed inclusions and sequential
colimits of levelwise closed inclusions.  In particular, the geometric
fixed point functor preserves homotopy colimits. The two functors
also have right-hand and left-hand relationships (respectively) to the fixed
point functor of spaces. The zeroth space (or $n$-th space) of the
categorical fixed point functor is the fixed point space of the zeroth
space (or $n$-th space) of the orthogonal $\G$-spectrum:
\[
X^{C}(\bR^{n}) = (X(\bR^{n}))^{C}.
\]
Likewise, the geometric fixed point functor of a suspension $\G$-spectrum
(or $n$-shift desus\-pension $\G$-spectrum) is the suspension
$\G$-spectrum (or $n$-shift desuspension $\G$-spectrum) of the fixed
point space 
\[
\Phi^{C} (F_{\bR^{n}}A) \iso F_{\bR^{n}} (A^{C})
\]
(using the notation of \cite[II.4.6]{MM}).

It is evident from the definitions that both the categorical fixed
point functor and the geometric fixed point functor have a canonical enrichment in
based spaces. 
For the construction of function spectra 
for cyclotomic spectra in Section~\ref{sec:model}, we need an
enrichment in orthogonal spectra of the geometric fixed point functors
for finite subgroups.  To obtain this, we now describe a natural 
transformation   
\begin{equation}\label{eq:enriched}
F^{\G}(X,Y)\to
F^{\G/C}(\Phi^{C}X,\Phi^{C}Y)\iso
F^{\G}(\rho^{*}\Phi^{C}X,\rho^{*}\Phi^{C}Y)
\end{equation}
Here the isomorphism $F^{\G/C}(\Phi^{C}X,\Phi^{C}Y)\iso
F^{\G}(\rho^{*}\Phi^{C}X,\rho^{*}\Phi^{C}Y)$ on the
right is induced from the evident space-level isomorphism induced by
the $(\#C)$-th root isomorphism $\rho \colon \G\to \G/C$ (the
unique orientation preserving isomorphism of compact connected
$1$-dimensional Lie groups).  The left-hand map, on the other hand,
arises as a direct consequence of the following theorem (as we explain
below).

\begin{thm}\label{thm:phicommute}
For $X$ an orthogonal $\G$-spectrum, $A$ a cofibrant
non-equi\-variant orthogonal spectrum, and $C$ a closed subgroup of
$\G$, the canonical natural map
\[
\Phi^{C}(X) \sma A\to \Phi^{C}(X \sma A)
\]
is an isomorphism.
\end{thm}

As the previous theorem is a special case of a new foundational
observation about equivariant orthogonal spectra that holds for any
compact Lie group $G$, we state and prove it in the more general
context in Appendix~\ref{app:geneqres}.

To deduce~\eqref{eq:enriched}, we recall that $F^{\G}(X,Y)$ is the
orthogonal $\G$-spectrum with $n$-th space $F^{\G}(X,Y)(\bR^{n})$ the space of
$\G$-equivariant maps from  $X\sma F_{\bR^{n}}S^{0}$ to $Y$.  As the geometric
fixed point functor is enriched in based spaces, we get an induced map
from $F^{\G}(X,Y)(\bR^{n})$ to the
space of $\G/C$-equivariant maps from $\Phi^{C}(X\sma F_{\bR^{n}}S^{0})$ to
$\Phi^{C}Y$, which Theorem~\ref{thm:phicommute} then identifies as the
$n$-th space of $F^{\G/C}(\Phi^{C}X,\Phi^{C}Y)$.  This then assembles
to the map of orthogonal spectra 
in~\eqref{eq:enriched}. 

Finally, for later use we need a new observation on iterated geometric
fixed point functors.  Using the $r$-th root isomorphism $\rho_{r}\colon
\G\to \G/C_{r}$ and the concomitant functor $\rho^{*}_{r}$ from
$\G/C_{r}$-spaces to $\G$-spaces, for \an $\G$-space $A$, we have a
canonical natural identification of $\G$-spaces
\[
\rho^{*}_{mn}A^{C_{mn}} = \rho^{*}_{m}(\rho^{*}_{n}A^{C_{n}})^{C_{m}}
\]
for $m,n \in \bN$.  Using the analogous functor from orthogonal
$\G/C_{r}$-spectra to $\G$-spectra, we have the following orthogonal
spectrum version of this natural transformation for the geometric
fixed point functors.

\begin{prop}\label{prop:logical}
For every $m,n \in \bN$, there is a canonical natural map of
orthogonal $\G$-spectra 
\[
c_{m,n} \colon \rho^{*}_{mn}\Phi^{C_{mn}}X\to
\rho^{*}_{m}\Phi^{C_{m}}(\rho^{*}_{n}\Phi^{C_{n}}X)
\]
making the following diagram commute
\[
\xymatrix@C+1pc{%
\rho^{*}_{mnp}\Phi^{C_{mnp}}X\ar[d]_{c_{m,np}}\ar[r]^{c_{mn,p}}
&\rho^{*}_{mn}\Phi^{C_{mn}}(\rho^{*}_{p}\Phi^{C_{p}}X)
\ar[d]^{c_{m,n}}\\
\rho^{*}_{m}\Phi^{C_{m}}(\rho^{*}_{np}\Phi^{C_{np}}X)
\ar[r]_-{\rho^{*}_{m}\Phi^{m}c_{n,p}}
&\rho^{*}_{m}\Phi^{C_{m}}(\rho^{*}_{n}\Phi^{C_{n}}(\rho^{*}_{p}\Phi^{C_{p}}X)).
}
\]
The map $c_{m,n}$ is an isomorphism when $X=F_{V}A$ or more generally,
when $X$ is cofibrant (q.v.~Theorem~\ref{thm:ossmod}).
\end{prop}

\begin{proof}
In order to construct $c_{m,n}$, it suffices to construct a natural
transformation of $\rho^{*}_{mn}\sJ^{C_{mn}}_{\G}$-spaces from
$\rho^{*}_{mn}\Fix^{C_{mn}} X$ to
$\rho^{*}_{m}\Phi^{C_{m}}(\rho^{*}_{n}\Phi^{C_{n}}X)$.  For this, it
suffices to construct a natural transformation of
$\rho^{*}_{mn}\sJ^{C_{mn}}_{\G}$-spaces from $\rho^{*}_{mn}\Fix^{C_{mn}}X$
to $\rho^{*}_{m}\Fix^{C_{m}}(\rho^{*}_{n}\Phi^{C_{n}}X)$, and for this, it
suffices to construct a natural transformation of
$\rho^{*}_{mn}\sJ^{C_{mn}}_{\G}$-spaces from $\rho^{*}_{mn}\Fix^{C_{mn}}X$
to $\rho^{*}_{m}(\Fix^{C_{m}}\rho^{*}_{n}(\Fix^{C_{n}}X))$.  This is induced by
the space-level identity
\[
\rho^{*}_{mn}(X(V))^{C_{mn}} = \rho^{*}_{m}(\rho^{*}_{n}(X(V))^{C_{n}})^{C_{m}}.
\]
For the diagram, we observe that both composites are ultimately
induced by the same space-level canonical natural isomorphism.  For
$X=F_{V}A$, both $\rho^{*}_{mn}\Phi^{C_{mn}}X$ and
$\rho^{*}_{m}\Phi^{C_{m}}(\rho^{*}_{n}\Phi^{n}X)$ are isomorphic to
$F_{V^{C_{mn}}}A^{C_{mn}}$ and it follows that $c_{m,n}$ is an
isomorphism from the universal property of $F_{V}$.  The isomorphism
for a general cofibrant $X$ then follows from the fact that the functors
$\rho^{*}_{r}$ and $\Phi^{C_{r}}$ preserve pushouts over cofibrations
and sequential colimits of cofibrations.
\end{proof}

\section{Review of the homotopy theory of orthogonal \texorpdfstring{$\G$}{T}-spectra}

In the previous section we reviewed the definition of 
orthogonal $\G$-spectra and the construction of the point-set fixed
point functors.  In this section, we give a brief review of relevant
aspects of the homotopy theory of orthogonal $\G$-spectra, including
the homotopy groups, the model structures, and the derived functors of
the fixed point functors. 

We begin with the level model structure on orthogonal $\G$-spectra,
which is mainly a tool for construction of the stable
model structure, but in the non-equivariant setting plays a role in
later sections allowing us to simplify hypotheses in certain
statements (see Theorems~\ref{thm:newFp} and~\ref{thm:newF} and the
arguments in Section~\ref{sec:corep}). 
In the level model structure, the weak equivalences are the
\term{level equivalences}, which are the maps $X\to Y$ that are
equivariant weak equivalences $X(V)\to Y(V)$ for all orthogonal
$\G$-representations $V$, or in other words, the maps that are
non-equivariant weak equivalences $X(V)^{C}\to Y(V)^{C}$ for all
closed subgroups $C\leq\G$.  The level fibrations are the maps that are
equivariant Serre fibrations $X(V)\to Y(V)$ for all $V$, or in
other words, the maps that are non-equivariant Serre fibrations
$X(V)^{C}\to Y(V)^{C}$ for all closed subgroups $C\leq\G$.  In
particular, every object is fibrant.  The
cofibrations are defined by the left lifting property, and are
precisely the retracts of relative cell complexes (defined using
sequential colimits, e.g.~\cite[5.4]{MMSS}) built out of the
$V$-desuspension $\G$-spectra of standard $\G$-space $n$-cells 
\[
F_{V}(\G/C \times S^{n-1})_{+}\to 
F_{V}(\G/C \times D^{n})_{+}
\]
for any orthogonal $\G$-representation $V$ and any $n\geq 0$, where
$D^{n}$ denotes the unit disk in $\bR^{n}$, $S^{n-1}$ its boundary,
and $C\leq\G$ a closed subgroup.
These cofibrations are also the cofibrations in the stable model
structure described next.

In the stable model structure, we define the weak equivalences in
terms of homotopy groups.  In fact, we describe
two different versions of homotopy groups, both of which define the
same weak equivalences. 
The homotopy groups of an orthogonal $\G$-spectrum are the homotopy
groups of the underlying $\G$-(pre-)spectrum, which can be defined
concretely as follows.  
Because of our emphasis on cyclotomic spectra and $TC$, we will work
in terms of the specific complete ``$\G$-universe'' (countable
dimensional equivariant real inner product space) usually used in this
context. (The specifics do not play a significant role here; rather,
we follow the notation and exposition of \cite[\S4]{BM4} as
closely as possible.)
Let
\[
U=\bigoplus_{n=0}^{\infty}\bigoplus_{r=1}^{\infty}\bC(n),
\]
where $\bC(0)$ denotes the complex numbers with trivial $\G$-action,
$\bC(1)$ denotes the complex numbers with the standard $\G$-action,
and $\bC(n)$ denotes the complex numbers with $\G$ acting through the
$n$th power map, all regarded as real vector spaces.
Here the inner product is induced by the standard ($\G$-invariant)
hermitian product on $\bC(n)$ and orthogonal direct sum.
Every orthogonal $\G$-representation is then isometric to a finite
dimensional $\G$-stable subspace of $U$.  Notationally, we write
$V<U$ to denote a finite dimensional $\G$-stable subspace of
$U$
and for $V<W<U$, we denote by $W-V$ the orthogonal complement
of $V$ in $W$.

\begin{defn}\label{defn:pi}
For an orthogonal $\G$-spectrum $X$, we define the \term{homotopy groups} by
\[
\pi_{q}^{C}X = 
\begin{cases}
\quad\displaystyle 
\lcolim_{V< U} \pi_{q}((\Omega^{V}X(V))^{C})&q\geq 0\\[1em]
\quad\displaystyle 
\lcolim_{\bC(0)^{-q}\leq V< U} \pi_{-q}((\Omega^{V-\bC(0)^{-q}}X(V))^{C})&q< 0\\
\end{cases}
\]
for $q\in \bZ$ and $C$ a closed subgroup of $\G$ \cite[III.3.2]{MM}.
\end{defn}

The expression above for the homotopy groups $\pi_{*}^{C}$ provides an
intrinsic construction of the homotopy groups of the
underlying non-equivariant spectrum of the right derived categorical
fixed point functor.  Analogously, we can construct the homotopy
groups of the underlying non-equivariant spectrum of the left derived
geometric fixed point functor.  We call these the ``geometric homotopy
groups'' and use the notation $\pi^{\Phi C}_{*}$ to emphasize
the analogy with $\pi_{*}^{C}$.  The geometric homotopy groups were
denoted as  $\rho^{H}_{q}$ in \cite[\S V.4]{MM}.

\begin{defn}\label{defn:rho}
For an orthogonal $\G$-spectrum $X$, we define the \term{geometric homotopy
groups} by
\[
\pi_{q}^{\Phi C}X = 
\begin{cases}
\quad\displaystyle 
\lcolim_{V< U} \pi_{q}(\Omega^{V^{C}}({X}(V)^{C}))&q\geq 0\\[1em]
\quad\displaystyle 
\lcolim_{\bC(0)^{-q}\leq V< U} \pi_{-q}(\Omega^{V^{C}-\bC(0)^{-q}}({X}(V)^{C}))&q< 0\\
\end{cases}
\]
for $q\in \bZ$, and $C$ a closed subgroup of
$\G$ \cite[V.4.8.(iii),V.4.12]{MM}.
\end{defn}

It is clear from the formula that $\pi_{q}^{\Phi C}(-)$ sends level
equivalences of orthogonal $\G$-spectra to isomorphisms of abelian
groups.  Since by \cite[V.4.12]{MM}, $\pi_{*}^{\Phi C}X\iso
\pi_{*}(\Phi^{C}X)$ when $X$ is level cofibrant, it follows that for any
$X$, $\pi_{*}^{\Phi C}X$ calculates the homotopy groups of the
underlying non-equivariant spectrum of the left derived geometric
fixed point functor applied to $X$,
\[
\pi_{*}^{\Phi C}X\iso \pi_{*}(\bL\Phi^{C}X).
\]

Both the homotopy groups and the geometric homotopy groups detect the
weak equivalences on the stable model structure on
orthogonal $\G$-spectra:  We define a
\term{weak equivalence} of orthogonal $\G$-spectra (or \term{stable
equivalence} when necessary to distinguish from other notions of weak
equivalence) to be a map that induces an isomorphism on all homotopy
groups, or equivalently (by~\cite[\S XVI.6.4]{May-Alaska}
and~\cite[VI.5.1 or V.4.17]{MM}), a map that induces an isomorphism on all
geometric homotopy groups.  The cofibrations in the stable model
structure are the same as the cofibrations in the level model
structure, and we define the fibrations by the right lifting property. By
\cite[III.4.8]{MM}, a map $X\to Y$ is a fibration in the stable model
structure exactly when it is a level fibration such that
the diagram
\[
\xymatrix@-1pc{
X(V) \ar[r] \ar[d] & \Omega^{W} X(V \oplus W) \ar[d] \\
Y(V) \ar[r] & \Omega^{W} Y(V \oplus W)
}
\]
is a homotopy pullback for all orthogonal $\G$-representations $V$,
$W$.  In particular, an object is cofibrant if and only if it is the
retract of a cell complex, and an object is fibrant if and only if it
is an equivariant $\Omega$-spectrum.

\begin{thm}[Stable model structure {\cite[4.2]{MM}}]\label{thm:ossmod}
The category of orthogonal $\G$-spectra is a cofibrantly-generated
closed model category in which a map $X \to Y$ is 
\begin{itemize}
\item A weak equivalence if the induced map on homotopy groups
$\pi_{q}^{C}$ is an isomorphism for all $q \in \bZ$ and all closed subgroups $C \leq  \G$, or
equivalently, 
if the induced map on geometric homotopy groups
$\pi_{q}^{\Phi C}$ is an isomorphism for all $q \in \bZ$ and all closed subgroups $C \leq  \G$.
\item A cofibration if it is a retract of a relative cell complex, and 
\item A fibration if it satisfies the right lifting property with
respect to the acyclic cofibrations.
\end{itemize}
Moreover, the model structure is compatible with the enrichment of
orthogonal $\G$-spectra over orthogonal spectra, meaning that the
analogue of Quillen's axiom SM7 is satisfied.
\end{thm}

For our purposes, we need model structures for some localized homotopy
categories.  We start with the homotopy categories local to a family.

\begin{defn}[{\cite[IV.6.1]{MM}}]\label{defn:localequiv}
Let $\aF$ be a family of subgroups of $\G$, i.e., a
collection of closed subgroups of $\G$ closed under taking closed
subgroups (and conjugation).  Let $X,Y$ be orthogonal $\G$-spectra.
An \term{$\aF$-local equivalence} (or \term{$\aF$-equivalence}) is a map
$X\to Y$ that induces an isomorphism on homotopy groups $\pi^{C}_{*}$
for all $C$ in $\aF$, or equivalently, on all
geometric homotopy groups $\pi^{\Phi C}_{*}$
for all $C$ in $\aF$.
\end{defn}

An \term{$\aF$-cofibration} is a map
built as a retract of a relative cell complex using cells 
\[
F_{V}(\G/C \times S^{n-1})_{+}\to 
F_{V}(\G/C \times D^{n})_{+},
\]
where we require $C\in \aF$.  The $\aF$-fibrations are then defined by
the right lifting property.  Explicitly, a map $X\to Y$ of orthogonal
$\G$-spectra is an $\aF$-fibration exactly when it is levelwise an
$\aF$-fibration of spaces (the maps $(X(V))^{C}\to (Y(V))^{C}$ are
non-equivariant Serre fibrations for each orthogonal
$\G$-representation $V$ and each $C\in
\aF$) such that the diagram
\[
\xymatrix@-1pc{
(X(V))^{C} \ar[r] \ar[d] & (\Omega^{W} X(V \oplus W))^{C} \ar[d] \\
(Y(V))^{C} \ar[r] & (\Omega^{W} Y(V \oplus W))^{C}
}
\]
is a homotopy pullback for all orthogonal $\G$-representations $V$,
$W$ and all $C \in \aF$. 

\begin{thm}[$\aF$-local model structure
{\cite[IV.6.5]{MM}}]\label{thm:osflocalmod} 
The category of orthogonal $\G$-spectra is a cofibrantly-generated
closed model category in which the weak equivalences, cofibrations, and
fibrations are the $\aF$-equivalences, $\aF$-cofibrations, and
$\aF$-fibrations, respectively. The model structure is compatible with
the enrichment of orthogonal $\G$-spectra over orthogonal spectra,
meaning that the analogue of Quillen's axiom SM7 is satisfied.
\end{thm}

For our corepresentability results, we use model structures based on
the finite complete or $p$-complete homotopy categories.  Letting
$M^{1}_{p}$ denote the mod-$p$ Moore space in dimension $1$, we define a
$p$-equivalence to be a map that becomes a weak equivalence after
smashing with $M^{1}_{p}$.  Likewise, we define a
$p$-$\aF$-equivalence to be a map that becomes an $\aF$-equivalence
after smashing with $M^{1}_{p}$.  (Note that since $M^{1}_{p}$ is a CW
complex, smash product with it preserves $\aF$-equivalences for any
family $\aF$.)  
We also have the more general notions of finite
equivalence and finite $\aF$-equivalence, which are the maps that are
$p$-equivalences and $p$-$\aF$-equivalences (resp.) for all $p$.  The
argument for \cite[IV.6.3]{MM}, applied directly starting with the
$\aF$-local model structure rather than the standard stable model
structure, proves the following theorem.

\begin{thm}[$\aF$-local finite complete model structure]%
\label{thm:ospflocalmod}
The category of orthogonal $\G$-spectra is a cofibrantly-generated
closed model category in which the weak equivalences are the finite
$\aF$-equivalences (resp., finite $p$-$\aF$-equivalences), and the
cofibrations are the $\aF$-cofibrations. The model structure is
compatible with the enrichment of orthogonal $\G$-spectra over
orthogonal spectra, meaning that the analogue of Quillen's axiom SM7
is satisfied with respect to the fibrations in the finite complete
(resp., $p$-complete) model category of orthogonal spectra.
\end{thm}

\begin{rem}
The model structures reviewed in this section all admit
explicit descriptions of the generating cofibrations and acyclic
cofibrations; for details, see the cited references where they are
constructed. 
\end{rem}

\section{The categories of cyclotomic and pre-cyclotomic spectra}\label{sec:def}

The work of the previous two sections provides the background we need
for the work in this section to define the categories of $p$-cyclotomic and 
cyclotomic spectra and for the work in the next
section to
construct model structures on these categories.  Here we
start with the easier category of $p$-cyclotomic spectra and then turn
to the more complicated category of cyclotomic spectra.

The definition of $p$-cyclotomic spectra requires both
extra structure on an orthogonal $\G$-spectrum and a homotopical
condition, which we break into separate pieces.

\begin{defn}[Pre-$p$-cyclotomic spectra]\label{defn:prepcyc}
A \term{pre-$p$-cyclotomic spectrum} $X$ is a pair $(X,\cyc)$
consisting of an orthogonal $\G$-spectrum $X$ together with a map of
orthogonal $\G$-spectra  
\[
\cyc \colon \rho_{p}^{*}\Phi^{C_{p}} X \to X,
\]
where $\rho_{p}$ is the $p$-th root isomorphism $\G\to \G/C_{p}$.
A morphism of pre-$p$-cyclotomic spectra $(X,t_X) \to (Y,t_Y)$
consists of a map of orthogonal $\G$-spectra $X\to Y$ such that the
diagram 
\[
\xymatrix@-0.5pc{%
\rho_{p}^{*}\Phi^{C_{p}}X\ar[r]^-{t_X}\ar[d]
&X\ar[d]\\
\rho_{p}^{*}\Phi^{C_{p}}Y\ar[r]^-{t_Y}
&Y
}
\]
commutes.
\end{defn}

To avoid unnecessary verbosity we will say simply ``cyclotomic maps'' rather than
``pre-$p$-cyclotomic maps'' when the context is clear.

Clearly the category of pre-$p$-cyclotomic spectra inherits an enrichment over
spaces, with the set of cyclotomic maps topologized using the
subspace topology from the space of maps of orthogonal $\G$-spectra.
In fact, the category of pre-$p$-cyclotomic spectra inherits an enrichment over
spectra.

\begin{prop}\label{prop:pcycenrich}
The category of pre-$p$-cyclotomic spectra inherits an enrichment over
orthogonal spectra from the enrichment on orthogonal $\G$-spectra and
the enrichment of the functor $\rho^{*}\Phi^{C_{p}}$.
\end{prop}

\begin{proof}
Using the canonical orthogonal spectrum enrichment on $\rho_{p}^{*}$
and the orthogonal spectrum enrichment on $\Phi^{C_{p}}$
from~\eqref{eq:enriched}, for orthogonal $\G$-spectra $X$ and $Y$ we
get the spectrum of cyclotomic maps 
\[
F_{Cyc}(X,Y) = \equalizer \relax[
\xymatrix@C-1pc{
F^{\G}(X,Y) \ar@<-.5ex>[r]\ar@<.5ex>[r]
&F^{\G}(\rho_{p}^{*}\Phi^{C_{p}}X,Y)
} ]
\]
formed as the equalizer of the map 
\[
F^{\G}(X,Y) \to 
F^{\G}(\rho_{p}^{*}\Phi^{C_{p}}X,Y)
\]
induced by the structure map for $X$ and the composite
\[
F^{\G}(X,Y) \to 
F^{\G}(\rho_{p}^{*}\Phi^{C_{p}}X,\rho_{p}^{*}\Phi^{C_{p}}Y)\to
F^{\G}(\rho_{p}^{*}\Phi^{C_{p}}X,Y)
\]
induced by~\eqref{eq:enriched} and the structure map for $Y$.  Because
equalizers are formed spacewise, the zeroth space of this mapping
spectrum is the space of cyclotomic maps from $X$ to $Y$.  Composition
in $F^{\G}$ induces composition on $F_{Cyc}$, which is compatible with
the composition of cyclotomic maps.  
\end{proof}

The category of pre-$p$-cyclotomic spectra is complete (has all
limits) but only has certain colimits; the natural map goes from the
colimit of the geometric fixed points to the geometric fixed points of
the colimit, and so we can typically only construct those colimits
where this map is an isomorphism.

\begin{prop}\label{prop:prepcyclimits}
The category of pre-$p$-cyclotomic spectra has all limits.  It has all
coproducts, pushouts along maps that are levelwise closed inclusions,
and sequential colimits of maps that are levelwise closed inclusions.
\end{prop}

\begin{proof}
Limits are created in the category of orthogonal $\G$-spectra, using
the natural map from the geometric fixed points of a limit to the
limit of the geometric fixed points as the structure map.  The natural
map from the colimit of the geometric fixed points to the
geometric fixed points of the colimit is an isomorphism for all of the
colimits in the statement since the space-level fixed point functor
commutes with these colimits.
\end{proof}

For indexed limits and colimits, similar observations apply.

\begin{prop}\label{prop:pcycorth}
For a cofibrant non-equivariant orthogonal spectrum $A$, $(-)\sma A$
extends to an endofunctor on pre-$p$-cyclotomic spectra that provides
the tensor with $A$ in the orthogonal spectrum enrichment of
pre-$p$-cyclotomic spectra.  For an arbitrary non-equivariant
orthogonal spectrum $A$, $F(A,-)$ extends to an endofunctor on
pre-$p$-cyclotomic spectra that provides the cotensor with $A$ in the
orthogonal spectrum enrichment of pre-$p$-cyclotomic spectra.
\end{prop}

\begin{proof}
For $X\sma A$, the structure map is induced by the structure map on
$X$ and the map 
\[
\rho_{p}^{*}\Phi^{C_{p}} (X\sma A)\to \rho_{p}^{*}\Phi^{C_{p}}X\sma A
\]
in Theorem~\ref{thm:phicommute}.  For $F(A,X)$, the structure map is
induced by the structure map on $X$ and the map
adjoint to the map
\[
\rho_{p}^{*}\Phi^{C_{p}}(F(A,X))\sma A
\to \rho_{p}^{*}\Phi^{C_{p}}(F(A,X)\sma A)
\to \rho_{p}^{*}\Phi^{C_{p}}X.
\]
An easy comparison of equalizers shows that $X\sma A$ is the tensor
and $F(A,X)$ is the cotensor of $X$ with $A$.
\end{proof}

The previous proposition in particular shows that the category of
pre-$p$-cyclotomic spectra has cotensors by all spaces and tensors by
cofibrant spaces.  In fact, the category of pre-$p$-cyclotomic has tensors
by all spaces since the smash product with spaces commutes with
geometric fixed points.

We define a $p$-cyclotomic spectrum to be a pre-$p$-cyclotomic
spectrum that satisfies the homotopical condition that the structure
map induces an $\aF_{p}$-equivalence in the equivariant stable
category 
\[
\rho_{p}^{*}\dPhi{C_{p}} X \to X,
\]
where $\dPhi{C_{p}}$ denotes the left derived functor of
$\Phi^{C_{p}}$ (see~\cite[V.4.5]{MM}), and $\aF_{p}$ denotes the family of $p$-groups,
$\aF_{p}=\{C_{p^{n}}\}$.   Since the geometric homotopy groups of
$\rho_{p}^{*}\dPhi{C_{p}}X$ are canonically isomorphic to the geometric homotopy groups of $X$~\cite[V.4.12]{MM}
\[
\pi^{\Phi C_{m}}_{*}(\rho_{p}^{*}\dPhi{C_{p}}X) \iso
\pi^{\Phi C_{mp}}_{*}(X),
\]
we can write this condition concisely as follows.

\begin{defn}[$p$-cyclotomic spectra]\label{defn:pcyc}
The category of \term{$p$-cyclotomic spectra} is the full subcategory of the
category of pre-$p$-cyclotomic spectra consisting of those objects $X$
for which the map $\pi^{\Phi C_{p^{n+1}}}_{q}(X)\to \pi^{\Phi
C_{p^{n}}}_{q}(X)$ induced by 
\[
\cyc(V) \colon \rho_{p}^{*}(X(V))^{C_{p}}\to X(V)
\] 
is an isomorphism for all $n\geq 0$, $q\in \bZ$.
\end{defn}

As a full subcategory, the category of $p$-cyclotomic spectra inherits
mapping spectra, and has the limits and colimits whose objects remain
in the category.  More specifically:

\begin{prop}\label{prop:pcyclimits}
The category of $p$-cyclotomic spectra has finite products and
pullbacks over fibrations.  It has all coproducts, pushouts over maps
that are levelwise closed inclusions, and sequential colimits over
levelwise closed inclusions.
\end{prop}

\begin{proof}
The assertion for colimits follows from the fact that the fixed point
functor on spaces preserves the colimits in the statement.  We can
deduce the existence of finite products and pullbacks over fibrations
from the fact that these homotopy limits are naturally weakly
equivalent to homotopy colimits.
\end{proof}

Because the derived geometric fixed point functor commutes with
derived smash product with non-equivariant spectra \cite[V.4.7]{MM},
$p$-cyclotomic 
spectra are closed under tensor (smash product) with cofibrant
non-equivariant orthogonal spectra.  Likewise, $p$-cyclotomic spectra
are closed under cotensor (function spectrum construction) with
cofibrant non-equivariant orthogonal spectra whose underlying object
in the stable category is finite.  

We now turn to pre-cyclotomic and cyclotomic spectra.  We require
structure maps for all primes $p$, with some compatibility relations.

\begin{defn}[Pre-cyclotomic spectra]\label{defn:precycc}
A \term{pre-cyclotomic spectrum} $X$ consists of an orthogonal $\G$-spectrum
$X$ together with structure maps 
\[
\cyc_{n}\colon \rho^{*}_{n}\Phi^{C_{n}}X\to X
\]
for all $n\geq 1$ such that the following diagram commutes for all
$m,n \in \bN$.
\[
\xymatrix@C+1.5pc{%
\rho^{*}_{mn}\Phi^{C_{mn}}X\ar[r]^-{\cyc_{mn}}\ar[d]_-{c_{m,n}}
&X\\
\rho^{*}_{m}\Phi^{C_{m}}(\rho^{*}_{n}\Phi^{C_{n}}X)
  \ar[r]_-{\rho^{*}_{m}\Phi^{C_{m}}\cyc_{n}}
&\rho^{*}_{m}\Phi^{C_{m}}X\ar[u]_-{\cyc_{m}}
}
\]
A map of pre-cyclotomic spectra is a map of orthogonal $\G$-spectra
that commutes with the structure maps.
\end{defn}

\begin{rem}
Clearly the structure of a pre-cyclotomic spectrum is determined by
the maps $\cyc_{p}$ for $p$ prime.  Vigleik Angeltveit has verified
that the relation
\begin{equation}\label{eq:vigeq}
\cyc_{p}\circ (\rho^{*}_{p}\cyc_{q})\circ c_{p,q}=
\cyc_{q}\circ (\rho^{*}_{q}\cyc_{p})\circ c_{q,p}
\end{equation}
for all primes $p$,$q$
implies the relation in the previous definition for all $m$,$n$.
Specifically, for any orthogonal $\G$-spectrum $X$, there is a
commutative diagram (for primes $p,q,r$) 
\[
\xymatrix@C-3pc{
& \rho^*_p \Phi^p \rho^*_q \Phi^q \rho^*_r \Phi^r X &
& \rho^*_q \Phi^q \rho^*_p \Phi^p \rho^*_r \Phi^r X & \\
& & \rho^*_{pq} \Phi^{pq} \rho^*_r \Phi^r X \ar[ul] \ar[ur] & & \\
& \rho^*_p \Phi^p \rho^*_{qr} \Phi^{qr} X \ar[uu] \ar[dl] &
& \rho^*_q \Phi^q \rho^*_{pr} \Phi^{pr}
X \ar[uu] \ar[dr] & \\
\rho^*_p \Phi^p \rho^*_q \Phi^q \rho^*_r \Phi^r X & & \rho^*_{pqr} \Phi^{pqr} X \ar[ul] \ar[uu] \ar[ur] \ar[dl]
\ar[dd] \ar[dr] & & \rho^*_q \Phi^q \rho^*_r \Phi^r \rho^*_p \Phi^p X \\
& \rho^*_{pr} \Phi^{pr} \rho^*_q \Phi^q X \ar[ul] \ar[dd] &
& \rho^*_{qr} \Phi^{qr} \rho^*_p \Phi^{p}
X \ar[ur] \ar[dd] & \\
& & \rho^*_r \Phi^{r} \rho^*_{pq} \Phi^{pq} X \ar[dl] \ar[dr] & & \\
& \rho^*_r \Phi^{r} \rho^*_p \Phi^{p} \rho^*_q \Phi^{q} X &
& \rho^*_r \Phi^{r} \rho^*_q \Phi^{q} \rho^*_p \Phi^{p} X. & \\
}
\]
If $X$ is equipped with pre-cyclotomic structure maps, then this
diagram yields six different maps $\rho^*_{pqr} \Phi^{pqr} X \to X$.
The relation in equation~\eqref{eq:vigeq} now implies (after a little
diagram-chasing) that these six maps are equal.
\end{rem}

\begin{defn}[Cyclotomic spectra]\label{defn:cyc}
The category of \term{cyclotomic spectra} is the full subcategory of the
category of pre-cyclotomic spectra consisting of those objects for
which the map $\pi^{\Phi C_{mn}}_{q}(X)\to \pi^{\Phi C_{m}}_{q}(X)$
induced by 
\[
\cyc_{n}(V)\colon \rho_{n}^{*}(X(V))^{C_{n}}\to X(V)
\] 
is an isomorphism for all $m,n\geq 1$, $q\in \bZ$.
\end{defn}

Once again, we get a spectrum of cyclotomic maps of pre-cyclotomic (or
cyclotomic) spectra as the equalizer 
\[
F_{Cyc}(X,Y) = \equalizer \relax\left[
\xymatrix@C-1pc{
F^{\G}(X,Y) \ar@<-.5ex>[r]\ar@<.5ex>[r]
&\prod\limits_{n \geq 1} F^{\G}(\rho_{n}^{*}\Phi^{C_{n}}X,Y)
} \right]
\]
of the maps determined by the maps
\[
F^{\G}(X,Y) \to 
F^{\G}(\rho_{n}^{*}\Phi^{C_{n}}X,Y)
\]
induced by the structure map for $X$ and the composites
\[
F^{\G}(X,Y) \to 
F^{\G}(\rho_{n}^{*}\Phi^{C_{n}}X,\rho_{n}^{*}\Phi^{C_{n}}Y)\to
F^{\G}(\rho_{n}^{*}\Phi^{C_{n}}X,Y)
\]
induced by~\eqref{eq:enriched} and the structure map for $Y$.
Propositions analogous to the ones above hold for the categories of
pre-cyclotomic spectra and cyclotomic spectra.

\begin{example}\label{exa:sphere}
The $S^1$-equivariant sphere spectrum has a canonical structure as a
cyclotomic spectrum induced by the canonical isomorphisms
$\rho_{n}^{*} \Phi^{C_{n}} S \cong S$.
\end{example}

We close by comparing the definition of cyclotomic spectra here to
definitions in previous work.  As far as we know the only definition
of a point-set category of cyclotomic spectra entirely in the context
of orthogonal $\G$-spectra is \cite[4.2]{BM4}
(cf.~\cite[\S1.2]{HMAnnals}), where the definition and construction of
$TC$ is compared with older definitions in the context of Lewis-May
spectra, e.g.,~\cite{HM2}.  In \cite{BM4} the authors lacked
Proposition~\ref{prop:logical} and so wrote a spacewise definition.
An easy check of universal properties reveals that the definition here
coincides with the definition there.

\section{Model structures on cyclotomic spectra}\label{sec:model}

We now move on to the homotopy theory of $p$-cyclotomic spectra and
cyclotomic spectra, which we express in terms of model structures.
The model structures are inherited from the ambient categories of
pre-$p$-cyclotomic and pre-cyclotomic spectra, where they are
significantly easier to set up, using standard arguments for
categories of algebras over monads.

\begin{cons}\label{cons:cell}
For an orthogonal $\G$-spectrum $X$, let 
\begin{align*}
\pC X &= X \vee \rho^{*}_{p}\Phi^{C_{p}} X \vee
\rho^{*}_{p}\Phi^{C_{p}}(\rho^{*}_{p}\Phi^{C_{p}} X)\vee \dotsb \\
\bC X &= \bigvee_{n\geq 1} \rho^{*}_{n}\Phi^{C_{n}}X, \,\, n\in \bZ.
\end{align*}
\end{cons}

The functor $\pC$ is a monad on the category of orthogonal
$\G$-spectra, the free monad generated by the endofunctor
$\rho^{*}_{p}\Phi^{C_{p}}$.  Clearly, the category of
pre-$p$-cyclotomic spectra is precisely the category of $\pC$-algebras
in orthogonal $\G$-spectra.
Because of the apparent failure of the
canonical map  
\[
\rho^{*}_{mn}\Phi^{C_{mn}}X\to
\rho^{*}_{n}\Phi^{C_{n}}(\rho^{*}_{m}\Phi^{C_{m}}X) 
\]
of Proposition~\ref{prop:logical} 
to be an isomorphism, $\bC$ does not appear to be a monad on the
category of orthogonal $\G$-spectra; however, it is a monad on the
full subcategory of cofibrant orthogonal $\G$-spectra, on which the
map is an isomorphism.  The unit is the inclusion of $X$ as
$\rho^{*}_{C_{1}} \Phi^{C_{1}}X$.  The 
multiplication is induced by the inverse isomorphisms
\[
\rho^{*}_{n}\Phi^{C_{n}}(\rho^{*}_{m}\Phi^{C_{m}}X)\to
\rho^{*}_{mn}\Phi^{C_{mn}}X.
\]
Pre-cyclotomic spectra with cofibrant underlying orthogonal
$\G$-spectra are precisely the $\bC$-algebras in the category of
cofibrant orthogonal $\G$-spectra.  More generally, 
every pre-cyclotomic spectrum comes with a canonical natural map $\xi
\colon \bC X\to X$ and a cyclotomic map is precisely a map of orthogonal
$\G$-spectra $f\colon X\to Y$ that makes the diagram
\[
\xymatrix{%
\bC X \ar[d]_{\xi}\ar[r]^{\bC f}&\bC Y\ar[d]^{\xi}\\
X\ar[r]_{f} & Y
}
\]
commute.  We cannot say much more, except for the following
proposition that allows us to treat $\bC X$ like a free pre-cyclotomic
spectrum functor.

\begin{prop}\label{prop:freeadj}
Let $A$ be a cofibrant orthogonal $\G$-spectrum.  Then $\pC A$ is a
pre-$p$-cyclotomic spectrum and $\bC A$ is a pre-cyclotomic spectrum
with structure maps induced by the monad multiplication.  If $X$ is a
pre-$p$-cyclotomic or pre-cyclotomic spectrum, then maps of orthogonal
$\G$-spectra from $A$ to $X$ are in one-to-one correspondence with
cyclotomic maps $\pC A\to X$ or $\bC A \to X$, respectively.
\end{prop}

The usual theory of model structures on algebra categories tells us to
define cells of pre-$p$-cyclotomic spectra and pre-cyclotomic spectra
using the free functor applied to cells in the model structure on the
underlying category, in this case, the $\aF_{p}$-local or
$\aF_{\fin}$-local model structure on orthogonal $\G$-spectra,
respectively (where $\aF_{p}$ is the family of $p$-subgroups and
$\aF_{\fin}$ is the family of finite subgroups of $\G$).
Specifically, for pre-$p$-cyclotomic spectra, the cells are 
\[
\pC F_{V}(\G/C \times S^{n-1})_{+}\to \pC F_{V}(\G/C \times D^{n})_{+}
\]
for $V$ an orthogonal $\G$-representation, $n\geq 0$, and $C<\G$ a
$p$-subgroup, and for pre-cyclotomic spectra, the cells are 
\[
\bC F_{V}(\G/C \times S^{n-1})_{+}\to \bC F_{V}(\G/C \times D^{n})_{+}
\]
for $V$ an orthogonal $\G$-representation, $n\geq 0$, and $C<\G$ a
finite subgroup.  The first model structure theorem is then the following.

\begin{thm}\label{thm:modelprecyc}
The category of pre-$p$-cyclotomic spectra has a cofibrantly-generated
model structure with 
\begin{itemize}
\item Weak equivalences the $\aF_{p}$-equivalences of the underlying orthogonal $\G$-spectra,
\item cofibrations the retracts of relative cell complexes built out
of the cells above, and
\item Fibrations the $\aF_{p}$-fibrations of the underlying orthogonal $\G$-spectra.
\end{itemize}
The category of pre-cyclotomic spectra has a cofibrantly-generated
model structure with 
\begin{itemize}
\item Weak equivalences the $\aF_{\fin}$-equivalences of the underlying orthogonal $\G$-spectra,
\item Cofibrations the retracts of relative cell complexes built out
of the cells above, and
\item Fibrations the $\aF_{\fin}$-fibrations of the underlying orthogonal $\G$-spectra.
\end{itemize}
Thus, pre-$p$-cyclotomic spectra and pre-cyclotomic spectra form
model* categories with the above model structures.
\end{thm}

\begin{proof}
As in \cite[5.13]{MMSS} (and~\cite[III\S8]{MM}), the model structure
statements follow from a ``Cofibration Hypothesis'' \cite[5.3]{MMSS}
about pushouts and sequential colimits.  Specifically, recall that a
map $X \to Y$ of orthogonal $\G$-spectra is an $h$-cofibration if it
satisfies the homotopy extension property.  The Cofibration Hypothesis
is satisfied for a collection of maps $\aI$ when the following two
conditions hold. 
\begin{enumerate}
\item Let $i \colon A \to B$ be a coproduct of maps in $\aI$.  In any
pushout 
\[
\xymatrix{
A \ar[r] \ar[d]^-i & X \ar[d]^-j \\
B \ar[r] & Y\\
}
\]
of pre-$p$-cyclotomic (or $p$-cyclotomic) spectra, the cobase change
$j$ is an $h$-cofibration of orthogonal $\G$-spectra.
\item The sequential colimit of a sequence of maps $f_i$ in
pre-$p$-cyclotomic (or $p$-cyclotomic) spectra that are
$h$-cofibrations of orthogonal $\G$-spectra is computed as the
sequential colimit in the category of orthogonal $\G$-spectra.
\end{enumerate}

In order to construct the model structures, it suffices to show that
the Cofibration Hypothesis holds for the candidate generating
cofibrations and acyclic cofibrations produced by applying $\bC_p$ and
$\bC$ to the generating cofibrations and acyclic cofibrations in the
$\aF_p$-local and $\aF_{\fin}$-local model structures on orthogonal
$\G$-spectra.  But this is clear from the fact that the geometric
fixed point functor preserves the colimits in question.

The last statement then follows from
Proposition~\ref{prop:prepcyclimits} and the corresponding proposition
for pre-cyclotomic spectra.
\end{proof}

Similarly, starting with the $\aF$-local $p$- and finite complete
model structures on orthogonal $\G$-spectra
(Theorem~\ref{thm:ospflocalmod}), we obtain the following ``complete''
model structures on pre-$p$-cyclotomic spectra and pre-cyclotomic
spectra.

\begin{thm}\label{thm:completemodelprecyc}
The category of pre-$p$-cyclotomic spectra has a cofibrantly-generated
model structure with 
\begin{itemize}
\item Weak equivalences the $p$-$\aF_{p}$-equivalences of the underlying orthogonal $\G$-spectra,
\item Cofibrations the retracts of relative cell complexes built out
of the cells above, and
\item Fibrations the fibrations of the underlying orthogonal
$\G$-spectra in the $\aF_{p}$-local $p$-complete model structure.
\end{itemize}
The category of pre-cyclotomic spectra has a cofibrantly-generated
model structure with 
\begin{itemize}
\item Weak equivalences the finite complete $\aF_{\fin}$-equivalences of the underlying orthogonal $\G$-spectra,
\item Cofibrations the retracts of relative cell complexes built out
of the cells above, and
\item Fibrations the fibrations of the underlying orthogonal
$\G$-spectra in the $\aF_{\fin}$-local finite complete model structure.
\end{itemize}
Thus, pre-$p$-cyclotomic spectra and pre-cyclotomic spectra form
model* categories with the above model structures.
\end{thm}

Turning to $p$-cyclotomic and cyclotomic spectra, 
because $\aF$-equivalences are defined in terms of the geometric
homotopy groups (Definition~\ref{defn:localequiv}), we have the
following simpler description of weak equivalences in this context.

\begin{prop}\label{prop:underlie}
A cyclotomic map of $p$-cyclotomic spectra is an $\aF_{p}$-equivalence of the
underlying orthogonal $\G$-spectra if and only if it is a weak
equivalence of the underlying non-equivariant orthogonal spectra.  
A cyclotomic map of cyclotomic spectra is an $\aF_{\fin}$-equivalence of the
underlying orthogonal $\G$-spectra if and only if it is a weak
equivalence of the underlying non-equivariant orthogonal spectra.  
\end{prop}

\begin{proof}
We give the argument for cyclotomic spectra; the proof for
$p$-cyclotomic spectra is analogous.  Clearly, a map of cyclotomic
spectra which is an $\aF_{\fin}$-equivalence of the underlying orthogonal
$\G$-spectra is a weak equivalence of underlying non-equivariant
orthogonal spectra.  Conversely, suppose we are given a map $X \to Y$
of $p$-cyclotomic spectra which is a weak equivalence of underlying
non-equivariant orthogonal spectra.  In the diagram
\[
\xymatrix{
\rho_{n}^{*}\bL\Phi^{C_{n}}X\ar[r]\ar[d]&
\rho_{n}^{*}\Phi^{C_{n}}X\ar[r]^-t\ar[d]
&X\ar[d]\\
\rho_{n}^{*}\bL\Phi^{C_{n}}Y\ar[r]&
\rho_{n}^{*}\Phi^{C_{n}}Y\ar[r]^-t
&Y, }
\]
the composite horizontal maps and the right-hand vertical map are weak
equivalences of underlying non-equivariant spectra, and so we conclude
that so is the left-hand vertical map.  Therefore, $\pi_*^{\Phi C_n} X
\to \pi_*^{\Phi C_n} Y$ is an isomorphism.  Inductively, we conclude
that $X \to Y$ is an $\aF_{\fin}$-equivalence of $\G$-spectra.
\end{proof}

Theorem~\ref{thm:modelstar}, the main theorem on the homotopy theory
of $p$-cyclotomic spectra, is now an immediate consequence.

\begin{proof}[Proof of Theorem~\ref{thm:modelstar}]
Any subcategory of a model* category that is closed under weak
equivalences, finite products and coproducts, pushouts over
cofibrations, and pullbacks over fibrations is itself a model*
category with the inherited model structure.  Clearly, the
$p$-cyclotomic spectra regarded as a subcategory of the model*
category of pre-$p$-cyclotomic spectra (with the
$\aF_{p}$-equivalences) satisfies these conditions.
Proposition~\ref{prop:underlie} now yields the characterization of the
weak equivalences given in the statement of the theorem.
\end{proof}

Similarly, Proposition~\ref{prop:underlie} also implies the
corresponding theorem for cyclotomic spectra.

\begin{thm}\label{thm:modelstarnotp}
The category of cyclotomic spectra is a model* category with weak
equivalences the weak equivalences of the underlying non-equivariant
orthogonal spectra and with fibrations the $\aF_{\fin}$-fibrations of
the underlying orthogonal $\G$-spectra.
\end{thm}

For the ``complete'' model structures, we should look at the closure
of the categories of $p$-cyclotomic spectra and cyclotomic spectra under
the weak equivalences in those model structures.  We refer to these as
``weak'' $p$-cyclotomic and cyclotomic spectra.

\begin{defn}\label{defn:weak}
The category of \term{weak $p$-cyclotomic spectra} is the full subcategory of the
category of pre-$p$-cyclotomic spectra consisting of those objects $X$
for which the map $\pi^{\Phi C_{p^{n+1}}}_{q}(X\sma M^{1}_{p})\to \pi^{\Phi
C_{p^{n}}}_{q}(X\sma M^{1}_{p})$ induced by 
\[
\cyc(V) \colon \rho_{p}^{*}(X(V)\sma M^{1}_{p})^{C_{p}}\to X(V)\sma M^{1}_{p}
\] 
is an isomorphism for all $n\geq 0$, $q\in \bZ$, where $M^{1}_{p}$ denotes the
mod-$p$ Moore space in dimension $1$.

The category of \term{weak cyclotomic spectra} is the full subcategory of the
category of pre-cyclotomic spectra consisting of those objects for
which the maps $\pi^{\Phi C_{mn}}_{q}(X\sma M^{1}_{p})\to \pi^{\Phi
C_{m}}_{q}(X\sma M^{1}_{p})$
are isomorphisms for all $m,n\geq 1$, $q\in \bZ$, and $p$ prime.
\end{defn}

Once again, the category of weak $p$-cyclotomic spectra and weak
cyclotomic spectra are suitable subcategories of model* categories.
Along with Proposition~\ref{prop:underlie}, this implies the following
result.

\begin{thm}\label{thm:weakmodel}
The category of weak $p$-cyclotomic spectra is a model* category with weak
equivalences the $p$-equivalences of the underlying non-equivariant
orthogonal spectra and with fibrations the $\aF_{p}$-local
$p$-complete fibrations of the underlying orthogonal $\G$-spectra.

The category of weak cyclotomic spectra is a model* category with weak
equivalences the finite equivalences of the underlying non-equivariant
orthogonal spectra and with fibrations the $\aF_{\fin}$-local
finite complete fibrations of the underlying orthogonal $\G$-spectra.
\end{thm}

Quillen's SM7 axiomatizes the compatibility between the model
structure and the enrichment.  In our context of an enrichment over
orthogonal spectra, the statement is the following theorem.

\begin{thm}\label{thm:SM7}
Let $i\colon W\to X$ be a cofibration of pre-$p$-cyclotomic (resp.,
pre-cyclotomic) spectra and let $f\colon Y\to Z$ be a fibration of
pre-$p$-cyclotomic (resp., pre-cyclotomic) spectra in either of the
model structures above.  Then the map
\[
F_{Cyc}(X,Y) \to F_{Cyc}(X,Z)\times_{F_{Cyc}(W,Z)}F_{Cyc}(W,Y)
\]
is a fibration of orthogonal spectra, and is a weak equivalence if
either $i$ or $f$ is.
\end{thm}

\begin{proof}
By the usual adjunction argument (using the adjunction of
Proposition~\ref{prop:pcycorth}), it is equivalent to show that for
every cofibration $j\colon A\to B$ of orthogonal spectra, the map
\[
F(A,Y) \to F(B,Z)\times_{F(A,Z)} F(B,Y)
\]
is a fibration of pre-$p$-cyclotomic (resp., pre-cyclotomic) spectra
and a weak equivalence whenever $j$ or $f$ is.  But since fibrations
of pre-$p$-cyclotomic (resp., pre-cyclotomic) spectra are just
$\aF_{p}$-fibrations (resp., $\aF_{\fin}$-fibrations) of the
underlying orthogonal $\G$-spectra, this is clear from the
corresponding fact in the category of orthogonal $\G$-spectra.
Note that in the $p$-complete (resp., finite complete) model structure,
we actually obtain a fibration in the $p$-complete (resp., finite
complete) model structure on orthogonal spectra.
\end{proof}

We proved SM7 using one of the usually equivalent adjoint
formulations.  The other adjoint formulation, called the
pushout-product axiom, is not equivalent in this context because not
all the relevant pushouts and tensors exist; however, the pushout-product axiom
does follow for those pushouts and
tensors that exist in the category.  Specifically, given a cofibration
of pre-$p$-cyclotomic (resp., pre-cyclotomic) spectra $j\colon X\to Y$ and a
cofibration of cofibrant
orthogonal spectra $i\colon A\to B$, the map  
\[
(Y\sma A)\cup_{(X\sma A)} (X\sma B)\to Y\sma B
\]
is a cofibration and a weak equivalence if either $i$ or $j$ is
cofibration.  An immediate consequence of this formula is the fact
that the model structures on pre-$p$-cyclotomic and pre-cyclotomic
spectra are stable in the sense that suspension (smash with $S^{1}$)
is an equivalence on the homotopy category with inverse equivalence
smash with $F_{\bR}S^{0}$.  As in~\cite[\S 7]{Hovey-ModelCat}, this
implies that the associated homotopy categories become triangulated
with the Quillen Puppe cofibration sequences defining the
distinguished triangles and the Quillen suspension defining the shift.  In
fact, Theorem~\ref{thm:SM7} directly gives the triangulated structure:
Mapping out of a cofibration sequence of cofibrant objects into a
fibrant object, we get a fibration sequence on the mapping spectra,
and mapping into a fibration sequence of fibrant objects from a
cofibrant object, we get a fibration sequence of mapping spectra.
Summarizing, we have the following proposition.

\begin{prop}
The homotopy categories of
pre-$p$-cyclotomic spectra and pre-cyclotomic spectra are
triangulated, with the distinguished triangles determined by the
cofiber sequences specified by the model structure
(see~\cite[6.2.6]{Hovey-ModelCat}) and suspension inducing the shift.
The homotopy categories of $p$-cyclotomic spectra 
and cyclotomic spectra are full triangulated subcategories of the
homotopy categories of pre-$p$-cyclotomic spectra and pre-cyclotomic
spectra, respectively.
\end{prop}

To compute the derived mapping spectrum using the enrichment, we take
the orthogonal spectrum of maps from a cofibrant cyclotomic spectrum
$X$ to a fibrant cyclotomic spectrum $Y$.
As cofibrant cyclotomic spectra do not typically arise in nature, we
offer the following construction that will allow us to construct the
derived mapping spectra using the more general class of cyclotomic
spectra whose underlying orthogonal $\G$-spectra are cofibrant.

\begin{cons}\label{cons:newFp}
For pre-$p$-cyclotomic spectra $X,Y$, let $F^{h}_{Cyc}(X,Y)$ be the
homotopy equalizer of the pair of maps $x_{p},y_{p}\colon F(X,Y)\to
F(\rho^{*}_{p}\Phi^{C_{p}}X,Y)$, where $x_{p}$ is induced by the
structure map for $X$ 
and $y_{p}$ is the composite 
\[
F^{\G}(X,Y) \to 
F^{\G}(\rho_{p}^{*}\Phi^{C_{p}}X,\rho_{p}^{*}\Phi^{C_{p}}Y)\to
F^{\G}(\rho_{p}^{*}\Phi^{C_{p}}X,Y)
\]
induced by~\eqref{eq:enriched} and the structure map for $Y$. 
Specifically, we construct $F^{h}_{Cyc}(X,Y)$ as the pullback 
\[
\xymatrix{%
F^{h}_{Cyc}(X,Y)\ar[r]\ar[d]
&F^{\G}(\rho_{p}^{*}\Phi^{C_{p}}X,Y)^{I}\ar[d]\\
F^{\G}(X,Y)\ar[r]_-{(x_{p},y_{p})}
&F^{\G}(\rho_{p}^{*}\Phi^{C_{p}}X,Y) \times
F^{\G}(\rho_{p}^{*}\Phi^{C_{p}}X,Y)
}
\]
where $(-)^{I}=F(I_{+},-)$ is the orthogonal spectrum of unbased maps
out of the unit interval and the vertical map is induced by the restriction to $\{0,1\}\subset I$. 
\end{cons}

\begin{thm}\label{thm:newFp}
Let $X$ and $Y$ be pre-$p$-cyclotomic spectra.  If $X$ is cofibrant, then
the canonical map from the equalizer to
the homotopy equalizer $F_{Cyc}(X,Y)\to
F^{h}_{Cyc}(X,Y)$ is a level equivalence.
\end{thm}

\begin{proof}
First consider the case when $X=\pC A$ for $A$ a $\aF_{p}$-cofibrant
orthogonal $\G$-spectrum.  In this case we are looking at the canonical
map from the pullback with $J=*$ to the pullback with $J=I$ (induced
by $I\to *$) in the following diagram.
\[
\xymatrix{%
&\displaystyle 
\prod_{n\geq 1}F^{\G}(\rho_{p_{n}}^{*}\Phi^{C_{p^{n}}}A,Y)^{J}\ar[d]\\
\displaystyle 
\prod_{n\geq 0}F^{\G}(\rho_{p_{n}}^{*}\Phi^{C_{p^{n}}}A,Y)\ar[r]
&\displaystyle 
\prod_{n\geq 1}F^{\G}(\rho_{p_{n}}^{*}\Phi^{C_{p^{n}}}A,Y)
\times 
\prod_{n\geq 1}F^{\G}(\rho_{p_{n}}^{*}\Phi^{C_{p^{n}}}A,Y).
}
\]
This is the map 
\[
F^{\G}(A,Y)\to 
F^{\G}(A,Y)\times_{\prod_{n\geq 1}F^{\G}(\rho_{p_{n}}^{*}\Phi^{C_{p^{n}}}A,Y)}
(\prod_{n\geq 1}F^{\G}(\rho_{p_{n}}^{*}\Phi^{C_{p^{n}}}A,Y))^{I_{+}},
\]
from $F^{\G}(A,Y)$ to the mapping path object, which is always a level
equivalence of orthogonal spectra.  Looking at pushouts over
cofibrations and sequential colimits over cofibrations, it follows
from \cite[III.2.7]{MM} that the map is a level equivalence for cell
pre-$p$-cyclotomic spectra.  Finally, since level equivalences are
preserved by retracts, it follows that the map is a level equivalence for cofibrant pre-$p$-cyclotomic spectra.
\end{proof}

\begin{cor}\label{cor:newFp}
Let $X$ be a pre-$p$-cyclotomic spectrum whose underlying orthogonal
$\G$-spectrum is $\aF_{p}$-cofibrant, and let $\overline X\to X$ be a
cofibrant replacement in the model* category of pre-$p$-cyclotomic
spectra.  Then for any fibrant pre-$p$-cyclotomic spectrum $Y$, the
maps
\[
F^{h}_{Cyc}(X,Y)\to F^{h}_{Cyc}(\overline X,Y)\from F_{Cyc}(\overline X,Y)
\]
are weak equivalences; thus, when $Y$ is fibrant, $F^{h}_{Cyc}(X,Y)$
represents the derived mapping spectrum.
\end{cor}

Note that if $X$ is just cofibrant as an orthogonal $\G$-spectrum,
$X\sma E\aF_{p+}$ is $\aF_{p}$-cofibrant; in practice, one expects to
apply the previous corollary to $X\sma E\aF_{p+}$ (which inherits a
canonical pre-$p$-cyclotomic structure from a pre-$p$-cyclotomic
structure on $X$).

We have an analogous construction in the context of pre-cyclotomic
spectra.  Because the definition of pre-cyclotomic spectra involves
compatibility relations, we need a homotopy limit over a more
complicated category.  As above, but for each $n>1$, we have maps 
\[
x_{n},y_{n}\colon F(X,Y)\to F(\rho^{*}_{n}\Phi^{C_{n}}X,Y)
\]
induced by the structure map on $X$ and the structure map on $Y$,
respectively.  More generally, we write $x_{n;m},y_{n;m}$ for the maps
\[
F(\rho^{*}_{m}\Phi^{C_{m}}X,Y)\to 
F(\rho^{*}_{mn}\Phi^{C_{mn}}X,Y)
\]
induced by the map $\rho^{*}_{m}\Phi^{C_{m}}\cyc_{n}$ on X,
\[
\rho^{*}_{mn}\Phi^{C_{mn}}X\to  
\rho^{*}_{m}\Phi^{C_{m}}(\rho^{*}_{n}\Phi^{C_{n}}X)
\overto{\rho^{*}_{m}\Phi^{C_{m}}\cyc_{n}} X,
\]
and
\[
F^{\G}(\rho^{*}_{m}\Phi^{C_{m}}X,Y) \to
F^{\G}(\rho_{n}^{*}\Phi^{C_{n}}(\rho^{*}_{m}\Phi^{C_{m}}X),\rho_{n}^{*}\Phi^{C_{n}}Y)\to
F^{\G}(\rho_{mn}^{*}\Phi^{C_{mn}}X,Y)
\]
induced by the map $\cyc_{n}$ on $Y$.  These maps satisfy the
following relations.
\begin{equation}\label{eq:xy}
x_{n;\ell m}\circ x_{m;\ell}=x_{mn;\ell}
\qquad
x_{n;\ell m}\circ y_{m;\ell}=y_{m;\ell n}\circ x_{n;\ell}
\qquad
y_{n;\ell m}\circ y_{m;\ell}=y_{mn;\ell}
\end{equation}
The first and last equations follows from the compatibility
requirement relating $\cyc_{mn}$ and $\cyc_{n}\circ \cyc_{m}$ (for $X$
and $Y$, respectively).  The middle equation follows from the
functoriality of $\rho^{*}_{m}\Phi^{C_{m}}$.

\begin{cons}\label{cons:newF}
Let $\Theta$ be the category whose objects are the positive integers
and maps freely generated by maps $x_{n;m},y_{n,m}\colon m\to mn$ for
all $m,n \in \bN$, subject to the relations~\eqref{eq:xy}.  Construct
$F^{h}_{Cyc}(X,Y)$ as the homotopy limit of the functor
$F^{\G}(\rho^{*}_{m}\Phi^{C_{m}}X,Y)$ from $\Theta$ to orthogonal
spectra. 
\end{cons}

The category $\Theta$ is (non-canonically) isomorphic to the opposite
of the category $\mathbb{I}$ used in the construction of $TC$
(see~\cite[3.1]{HM2} or Definition~\ref{defn:TC} below).

To explain the relationship of this construction with
Construction~\ref{cons:newFp}, let $\Theta_{p}$ be the full
subcategory of $\Theta$ consisting of the objects $p^{s}$ for $s\geq 0$.
The inclusion of the full subcategory consisting of $1$ and $p$
induces a map to the homotopy equalizer of $x_{p;1},y_{p;1}$ from
the homotopy limit over $\Theta_{p}$.

\begin{prop}\label{prop:Thetap}
Let $F$ be any functor from $\Theta_{p}$ to orthogonal spectra.  The
canonical map to the homotopy equalizer of $x_{p;1},y_{p;1}$ from
the homotopy limit over $\Theta_{p}$ is a level equivalence.
\end{prop}

\begin{proof}
By \cite[XI.9.1]{BousfieldKan}, it suffices to see that the inclusion
of $\{1,p\}$ into $\Theta_{p}$ is left cofinal, that is, for every
$s$, the category of maps in $\Theta_{p}$ from $\{1,p\}$ to $p^{s}$
has weakly contractible nerve.  The maps from $1$ to $p^{s}$ are in
one-to-one correspondence with monomials of the form $x^{i}y^{s-i}$
for $0\leq i\leq s$ and maps from $p$ to $p^{s}$ are in one-to-one
correspondence with monomials of the form $x^{i}y^{s-i-1}$ for $0\leq
i\leq s-1$.  For $x^{s}$ and $y^{s}$, there is exactly one
non-identity map in the category, namely $x^{s}\to x^{s-1}$ and
$y^{s}\to y^{s-1}$, respectively.  For every other monomial
$x^{i}y^{s-i}$, there are exactly two maps, $x_{p;1}$ and $y_{p;1}$,
to $x^{i-1}y^{s-i}$ and $x^{i}y^{s-i-1}$, respectively.  The nerve of
this category is therefore a generalized interval with $s+1$
$1$-simplices.
\end{proof}

Observe that the limit of the functor
$F^{\G}(\rho^{*}_{m}\Phi^{C_{m}}X,Y)$ from $\Theta$ to orthogonal
spectra is precisely $F_{Cyc}(X,Y)$.  As in theorem~\ref{thm:newFp}, the
canonical map from the limit to the homotopy limit is a level equivalence
when $X$ is cofibrant and therefore
$F^{h}_{Cyc}(X,Y)$ provides an
explicit model of the derived mapping space in the category of
pre-cyclotomic spectra when $Y$ is fibrant.

\begin{thm}\label{thm:newF}
Let $X$ and $Y$ be pre-cyclotomic spectra.  If $X$ is cofibrant, then
the canonical map from the limit to the homotopy limit
$F_{Cyc}(X,Y)\to F^{h}_{Cyc}(X,Y)$ is a level equivalence.
\end{thm}

\begin{proof}
By the usual argument, this reduces to the case of the domain and
codomain of a cell in pre-cyclotomic spectra, $X=\bC F_{V}(G/C_{n}\times
B)_{+}$ where $B=S^{q-1}$ or $D^{q}$.  Let $P=\{p_{1},\dotsc,p_{r}\}$ be the
distinct prime factors of $n$, and let $\Theta_{\hat P}$ be the full
subcategory of $\Theta$ consisting of the integers not divisible by
the primes in $P$.  Then we have that 
\[
\Theta = \Theta_{p_{1}}\times \dotsb \times \Theta_{p_{r}}\times
\Theta_{\hat P},
\]
and the Fubini theorem for homotopy limits gives a level equivalence 
\[
F^{h}_{Cyc}(X,Y)\simeq 
\holim_{\Theta_{p_{1}}}\dotsb \holim_{\Theta_{p_{r}}}
\holim_{\Theta_{\hat P}}F^{\G}(\rho^{*}_{m}\Phi^{C_{m}}X,Y),
\]
compatibly with the map from the limit.  Since $\Phi^{C_{m}}X=*$ for
all $m$ in $\Theta_{\hat P}$ except $m=1$, the equivalence above
reduces to a level equivalence 
\[
F^{h}_{Cyc}(X,Y)\simeq 
\holim_{\Theta_{p_{1}}}\dotsb \holim_{\Theta_{p_{r}}}
F^{\G}(\rho^{*}_{m}\Phi^{C_{m}}X,Y).
\]
Using the previous proposition, we can identify this homotopy limit as
an iterated homotopy equalizer.  As in the proof of
Theorem~\ref{thm:newFp}, we can then identify the map in question as
the map from $F_{V}(\G/C_{n}\times B)_{+}$ to an iterated mapping
path object.
\end{proof}

The analogue of Corollary~\ref{cor:newFp} also holds in this context: 

\begin{cor}
Let $X$ be a pre-cyclotomic spectrum whose underlying orthogonal
$\G$-spectrum is $\aF_{\fin}$-cofibrant, and let $\overline X\to X$ be a
cofibrant replacement in the model* category of pre-cyclotomic
spectra.  Then for any fibrant pre-cyclotomic spectrum $Y$, the
maps
\[
F^{h}_{Cyc}(X,Y)\to F^{h}_{Cyc}(\overline X,Y)\from F_{Cyc}(\overline X,Y)
\]
are weak equivalences; thus, $F^{h}_{Cyc}(X,Y)$ represents the derived
mapping spectrum when $Y$ is fibrant.
\end{cor}

\section{Corepresentability of \texorpdfstring{$TR$}{TR} and \texorpdfstring{$TC$}{TC}}\label{sec:corep}

We apply the work of the previous section to deduce corepresentability
results for $TC$ and related functors.  We begin with a brief review
of the constructions, starting with the ``restriction'' and
``Frobenius'' maps. 

Following \cite{MM}, we omit notation for the forgetful functor from
orthogonal $\G$-spectra to non-equivariant orthogonal spectra.

\begin{defn}\label{defn:R}
Let $X$ be cyclotomic spectrum. The
\term{restriction map} $R_{n}$ is the map of non-equivariant orthogonal spectra
\[
R_{n}\colon X^{C_{mn}}\to
(\rho^{*}_{n}\Phi^{C_{n}})^{C_{m}}\to
X^{C_{m}}
\]
induced by $\cyc_{n}$.  The \term{Frobenius map} $F_{n}$ is the map of 
non-equivariant orthogonal spectra
\[
F_{n}\colon X^{C_{mn}}\to X^{C_{m}}
\]
induced by the inclusion of fixed points.
\end{defn}

By convention, $R_{1}=F_{1}=\id\colon X\to X$.  It is easy to see that
$R_{n}$ and $F_{n}$ satisfy the following formulas.
\[
R_{m}R_{n}=R_{mn}\qquad F_{m}R_{n}=R_{n}F_{m}\qquad F_{m}F_{n}=F_{mn}
\]
We let $\mathbb{I}$ be the category with objects the positive natural
numbers and maps $R_{n}, F_{n} \colon nm \to m$ subject to the
relations above.  We write $R$ for the subcategory consisting of all the objects
and the restriction maps, and $F$ for the subcategory consisting of
all the objects and the Frobenius maps.

A cyclotomic spectrum $X$ determines a functor $m\mapsto X^{C_{m}}$ from 
$\mathbb{I}$ to non-equivariant orthogonal spectra.

\begin{defn}\label{defn:TC}
For a fibrant cyclotomic spectrum $X$, let 
\begin{align*}
TR(X)&=\holim_{R} X^{C_{m}} \\
TF(X)&=\holim_{F} X^{C_{m}} \\
TC(X)&=\holim_{\bI} X^{C_{m}}.
\end{align*}
\end{defn}

We have defined $TR$, $TF$, and $TC$ as point-set functors on the
subcategory of fibrant objects.  By taking fibrant approximations, we
obtain right derived functors $TR$, $TF$, and $TC$ from the homotopy
category of cyclotomic spectra to the stable category.

For a $p$-cyclotomic spectrum, we have the maps $R_{p}$, $F_{p}$, and
the corresponding full subcategories $R_{p}$, $F_{p}$, and $\bI_{p}$ of
$R$, $F$, and $\bI$, respectively, consisting of the subset of objects $\{1,p,p^{2},\dotsc\}$.

\begin{defn}
For a fibrant $p$-cyclotomic spectrum $X$, let 
\begin{align*}
TR(X;p)&=\holim_{R_{p}} X^{C_{p^{m}}} \\
TF(X;p)&=\holim_{F_{p}} X^{C_{p^{m}}} \\
TC(X;p)&=\holim_{\bI_{p}} X^{C_{p^{m}}}.
\end{align*}
\end{defn}

We now explain how to realize $TR$ and $TC$ as suitable mapping
objects in cyclotomic spectra.

\begin{cons}\label{cons:STRp}
Let $S_{TR;p}$ be the $p$-cyclotomic spectrum with underlying
orthogonal $\G$-spectrum 
\[
S_{TR;p}=\bigvee_{s\geq 0} F_{0}(\G/C_{p^{s}})_{+}
\]
and structure map from 
\begin{multline*}
\rho^{*}_{p}\Phi^{C_{p}}S_{TR;p} \cong
\bigvee_{s\geq 0} \rho^{*}_{p}\Phi^{C_{p}}(F_{0}(\G/C_{p^{s}})_{+})
\cong \bigvee_{s\geq 1} F_{0}(\G/C_{p^{s-1}})_{+}
=\bigvee_{s\geq 0} F_{0}(\G/C_{p^{s}})_{+}
\end{multline*}
to $S_{TR;p}$ induced by the canonical isomorphism.
\end{cons}

For any orthogonal $\G$-spectrum $X$, we have a canonical natural
isomorphism 
\[
F^{\G}(S_{TR;p},X)\iso\prod_{s\geq 0}X^{C_{p^{s}}}.
\]
For $X$ a $p$-cyclotomic spectrum,  $F^{h}_{Cyc}(S_{TR;p},X)$ is then
the homotopy equalizer
\[
F^{h}(S_{TR;p},X) \cong \hoequalizer 
\relax\left[
\xymatrix@C-1pc{
\prod X^{C_{p^{s}}} \ar@<-.5ex>[r]\ar@<.5ex>[r]
&\prod X^{C_{p^{s}}}
} \right]
\]
where one map is the identity (induced by the $p$-cyclotomic structure
map on $S_{TR;p}$) and the other is the product of the maps
$R_{p}\colon X^{C_{p^{s}}}\to X^{C_{p^{s-1}}}$.  This construction is
the ``mapping microscope'' of the maps $R_{p}$, which is a model for
the homotopy limit over $R_{p}$~\cite[2.2.8]{MayPonto}.  Since the
underlying orthogonal $\G$-spectrum of $S_{TR;p}$ is
$\aF_{p}$-cofibrant, we obtain the following theorem as a corollary of
Theorem~\ref{thm:newFp}.

\begin{thm}\label{thm:TRprep}
The right derived functor of $TR(-;p)$ is corepresentable in the homotopy
category of $p$-cyclotomic spectra, with corepresenting object $S_{TR;p}$.
\end{thm}

For $TC(-;p)$, we begin with the following pre-$p$-cyclotomic spectrum.

\begin{cons}\label{cons:STCp}
Let $S_{TC^{s};p}$ (for $s>0$) be the pre-$p$-cyclotomic spectrum whose
underlying orthogonal $\G$-spectrum is $F_{0}(\G/C_{p^{s}})_{+}$ and
whose structure map is
\[
\rho^{*}_{p}\Phi^{C_{p}}S_{TC^{s};p} \cong F_{0}(\G/C_{p^{s-1}})_{+}\to
F_{0}(\G/C_{p^{s}})_{+}=S_{TC^{s};p}
\]
induced by the quotient map $\G/C_{p^{s-1}}\to\G/C_{p^{s}}$.
The map $S_{TC^{s};p}\to S_{TC^{s+1};p}$ induced by the quotient is
then a map of pre-$p$-cyclotomic spectra; let $S_{TC;p}$ be the
telescope. 
\end{cons}

We can identify $TC(-;p)$ in terms of maps of pre-$p$-cyclotomic
spectra out of $S_{TC;p}$.

\begin{thm}\label{thm:STCp}
For pre-$p$-cyclotomic spectra $X$, there is a natural 
level equivalence
\[
\holim_{\bI_{p}}X^{C_{p^{m}}}\overto{\simeq}
F^{h}_{Cyc}(S_{TC;p},X).
\]
Thus, if $X$ is a fibrant $p$-cyclotomic spectrum, there is a natural
level equivalence
\[
TC(X;p)\overto{\simeq}F^{h}_{Cyc}(S_{TC;p},X).
\]
\end{thm}

\begin{proof}
Since $S_{TC;p}$ is the telescope of $S_{TC^{s};p}$, commuting
homotopy limits, we have that $F^{h}_{Cyc}(S_{TC;p},X)$ is the mapping
microscope of the orthogonal spectra $F^{h}_{Cyc}(S_{TC^{s};p},X)$.
Since $F^{\G}(S_{TC^{s};p},X)$ is canonically isomorphic to
$X^{C_{p^{s}}}$ and\break
$F^{\G}(\rho^{*}_{p}\Phi^{C_{p}}S_{TC^{s};p},X)$ is canonically
isomorphic to $X^{C_{p^{s-1}}}$, 
$F^{h}_{Cyc}(S_{TC^{s};p},X)$ is canonically isomorphic to a
homotopy equalizer of the form
\[
F^{h}_{Cyc}(S_{TC^{s};p},X) = \hoequalizer 
\relax\left[
\xymatrix@C-1pc{
X^{C_{p^{s}}} \ar@<-.5ex>[r]\ar@<.5ex>[r]
&X^{C_{p^{s-1}}}
} \right],
\]
with the maps induced by the structure map on $S_{TR^{s};p}$ and the
structure map on $X$.  The structure map on $X$ induces the map
$R_{p}\colon X^{C_{p^{s}}}\to X^{C_{p^{s-1}}}$ and the structure map
on $S_{TR^{s};p}$ induces the map $F_{p}\colon X^{C_{p^{s}}}\to
X^{C_{p^{s-1}}}$.  Therefore, we have a natural isomorphism 
\[
F^{h}_{Cyc}(S_{TC;p},X) \iso 
 \Mic_s \hoequalizer 
\relax\left[
\xymatrix@C-1pc{
X^{C_{p^{s}}} \ar@<-.5ex>[r]\ar@<.5ex>[r]
&X^{C_{p^{s-1}}}
} \right]
\]
where the maps on the homotopy equalizers are induced by the inclusion
of fixed points, i.e., by the maps $F_{p}$.  It is a standard fact in
$TC$ theory that the above homotopy limit is level equivalent to the
homotopy limit $\holim_{\bI_{p}}X^{C_{p^{m}}}$; we now review the
argument.  Let $\bI_{p}^{\leq s}$ denote the full subcategory of
$\bI_{p}$ consisting of the objects $m$ for $m\leq s$ and let
$\bI_{p}^{\{s-1,s\}}$. denote the full subcategory of $\bI_{p}$
consisting of the objects $s$ and $s-1$.  The argument of
Proposition~\ref{prop:Thetap} then shows that the inclusion of
$\bI_{p}^{\{s-1,s\}}$ in $\bI_{p}^{\leq s}$ induces a level
equivalence
\[
\holim_{\bI_{p}^{\leq s}}X^{C_{p^{m}}}\to
\holim_{\bI_{p}^{\{s-1,s\}}}X^{C_{p^{m}}}=
\hoequalizer 
\relax\left[
\xymatrix@C-1pc{
X^{C_{p^{s}}} \ar@<-.5ex>[r]\ar@<.5ex>[r]
&X^{C_{p^{s-1}}}
} \right].
\]
For fixed $s\geq 0$, the diagram
\[
\xymatrix@-1pc{%
\holim_{\bI_{p}^{\leq s+1}}X^{C_{p^{m}}}\ar[r]\ar[d]
&\holim_{\bI_{p}^{\{s,s+1\}}}X^{C_{p^{m}}}\ar[d]
\\
\holim_{\bI_{p}^{\leq s}}X^{C_{p^{m}}}\ar[r]
&\holim_{\bI_{p}^{\{s-1,s\}}}X^{C_{p^{m}}}
}
\]
commutes up to a canonical natural homotopy, where the vertical map
on the left is induced by the inclusion of $\bI_{p}^{\leq s}$ in
$\bI_{p}^{\leq s+1}$ and the
vertical map on the right is the map in the homotopy limit system
for $F_{Cyc}(S_{TC;p},X)$.  To see this, note that the down-then-right
map is induced by the natural inclusion $i$ of $\bI_{p}^{\{s-1,s\}}$
in $\bI_{p}^{\leq s+1}$ and the right-then-down map is induced by the
functor $j\colon \bI_{p}^{\{s-1,s\}}\to \bI_{p}^{\leq s+1}$ that sends
$s-1$ to $s$ and $s$ to $s+1$ (with the morphisms $F_{p}$ and $R_{p}$
from $s$ to $s-1$ going to the corresponding morphisms from $s+1$ to
$s$) together with $F_{p}$ viewed as a natural transformation to
the functor $X^{C_{p^{m}}}$ on $\bI_{p}^{\{s-1,s\}}$ from $j^{*}$ of the
functor $X^{C_{p^{m}}}$ on $\bI_{p}^{\leq s+1}$.  The maps $F_{p}$
then become a natural transformation from $i$ to $j$ and induce a
homotopy making the diagram commute, naturally in $X$.  Incorporating
these canonical natural homotopies, we get a canonical map
\[
\Mic_{s}\holim_{\bI_{p}^{\leq s}}X^{C_{p^{m}}}
\to
\Mic_{s}\holim_{\bI_{p}^{\{s-1,s \}}}X^{C_{p^{m}}}
\iso
F_{Cyc}(S_{TC;p},X)
\]
that is a level equivalence and natural in $X$.  The composite
\[
\holim_{\bI_{p}}X^{C_{p^{m}}}\iso
\lim\nolimits_{s}\holim_{\bI_{p}^{\leq s}}X^{C_{p^{m}}}
\overto{\simeq}
\Mic_{s}\holim_{\bI_{p}^{\leq s}}X^{C_{p^{m}}}
\overto{\simeq}
F_{Cyc}(S_{TC},X)
\]
is then the natural level equivalence in the statement.
\end{proof}

The preceding theorem now leads to the following characterization.

\begin{thm}\label{thm:weakrep}
For $X$ a fibrant weak $p$-cyclotomic spectrum, $TC(X;p)$ is naturally
weakly equivalent to $F^{h}_{Cyc}(S\sma E\aF_{p+},X)$, which represents
the derived mapping spectrum $\bR F_{Cyc}(S,X)$. 
\end{thm}

\begin{proof}
We have compatible canonical maps from $S_{TC^{s};p}$ to the sphere
spectrum $S$, induced by the collapse map $\G/C_{p^{s}}\to *$.  Using
the canonical cyclotomic structure on $S$ of Example~\ref{exa:sphere},
we obtain a canonical map of pre-$p$-cyclotomic spectra from
$S_{TC;p}$ to $S$.  Looking at mod $p$ homology, we see that the
unique map from the telescope of $\G/C_{p^{s}}$ to a point is a 
$p$-equivalence, and it follows that the map of pre-$p$-cyclotomic
spectra $S_{TC;p}\to S$ is a $p$-$\aF_{p}$-equivalence of
the underlying orthogonal $\G$-spectra.  We then have a zigzag of 
$p$-$\aF_{p}$-equivalences
\[
S_{TC;p}\from S_{TC;p}\sma E\aF_{p+}\to S\sma E\aF_{p+}
\]
where all of the above weak $p$-cyclotomic spectra have underlying
orthogonal $\G$-spectra that are cofibrant in both the $\aF_{p}$-local
model structure and in the $\aF_{p}$-local $p$-complete model
structure.  Since we have assumed that $X$ is fibrant in the
$\aF_{p}$-local $p$-complete model structure, the functors $F^{\G}(-,X)$
and $F^{\G}(\rho_{p}^{*}\Phi^{C_{p}}(-),X)$ convert the
$p$-$\aF_{p}$-equivalences of weak $p$-cyclotomic spectra above to
weak equivalences of orthogonal spectra.  Thus, we obtain a zigzag of
weak equivalences
\[
F^{h}_{Cyc}(S_{TC;p},X)\to F^{h}_{Cyc}(S_{TC;p}\sma E\aF_{p+},X)
\from F^{h}_{Cyc}(S\sma E\aF_{p+},X).
\]
Theorem~\ref{thm:STCp} gives a weak equivalence with $TC(X;p)$ and
Corollary~\ref{cor:newFp} shows that $F^{h}_{Cyc}(S\sma E\aF_{p+},X)$
represents the derived mapping spectrum $\bR F_{Cyc}(S,X)$. 
\end{proof}

We also have the following standard relationship between
$p$-completion and mapping spectra.

\begin{lem}\label{lem:pcompF}
Let $X$ be a fibrant $p$-cyclotomic spectrum and $X\phat$ a fibrant
replacement for $X$ in the category of weak $p$-cyclotomic spectra.
For any pre-$p$-cyclotomic spectrum $Y$ that is $\aF_{p}$-cofibrant as an
orthogonal $\G$-spectrum, $F^{h}_{Cyc}(Y,X)\to F^{h}(Y,X\phat)$ is a
$p$-equivalence. 
\end{lem}

\begin{proof}
Because the orthogonal $\G$-spectra $Y$ and $\rho_{p}^{*}\Phi^{C_{p}}Y$ are
$\aF_{p}$-cofibrant and the orthogonal $\G$-spectra $X$ and $X\phat$
are $\aF_{p}$-fibrant, the non-equivariant orthogonal spectra $F^{\G}(Y,X)$,
$F^{\G}(\rho^{*}_{p}\Phi^{C_{p}}Y,X)$, $F^{\G}(Y,X\phat)$, and
$F^{\G}(\rho^{*}_{p}\Phi^{C_{p}}Y,X\phat)$ are all fibrant.
For fibrant
orthogonal spectra, a map is a
$p$-equivalence if and only if the induced map on $F(M^{1}_p,-)$ is a
weak equivalence.   It follows that the maps
\[
F^{\G}(Y,X)\to 
F^{\G}(Y,X\phat)\qquad \text{and}\qquad 
F^{\G}(\rho^{*}_{p}\Phi^{C_{p}}Y,X)\to
F^{\G}(\rho^{*}_{p}\Phi^{C_{p}}Y,X\phat)
\]
are $p$-equivalences, and that the induced map on homotopy equalizers
\[ F^{h}_{Cyc}(Y,X)\to F^{h}(Y,X\phat) \]
is a $p$-equivalence.
\end{proof}

Finally, to prove Theorem~\ref{thm:main}, we need the following standard fact
about the relationship between $TC(-)\phat$ and $TC(-;p)$, rewritten
in the terminology of Section~\ref{sec:model}. 

\begin{thm}\label{thm:tcptc}
Let $X$ be a fibrant $p$-cyclotomic spectrum and $X\phat$ a fibrant
replacement for $X$ in the category of weak $p$-cyclotomic spectra.
Then the canonical map $TC(X;p)\to TC(X\phat;p)$ is a $p$-equivalence
and $TC(X\phat;p)$ is fibrant in the $p$-complete model structure on
orthogonal spectra.
\end{thm}

\begin{proof}
The $p$-equivalence statement follows from the previous lemma and
Theorem~\ref{thm:STCp}.
To see that $TC(X\phat;p)$ is fibrant in the $p$-complete model structure on
orthogonal spectra, we note that each
$(X\phat)^{C_{p^{s}}}=F^{\G}(F_{0}(\G/C_{p^{s}})_{+},X\phat)$ is
fibrant in the $p$-complete model structure of orthogonal spectra and
homotopy limits of fibrant objects are fibrant~\cite[18.5.2]{Hirschhorn}.
\end{proof}

Theorem~\ref{thm:main} is now an immediate consequence.

\begin{proof}[Proof of Theorem~\ref{thm:main}]
Theorem~\ref{thm:main} follows directly from the statements above, but to 
illustrate how the constructions fit together, we write out the
argument as a whole. Given a $p$-cyclotomic spectrum $X$, let
$\tilde{X}$ be a fibrant replacement in the model* category of
$p$-cyclotomic spectra and let $X\phat$ be a fibrant replacement of
$\tilde{X}$ in the model* category of weak $p$-cyclotomic spectra;
then $X\phat$ is also a fibrant replacement for $X$ in the model*
category of weak $p$-cyclotomic spectra.  We then have a commutative
diagram
\[
\xymatrix@R-1pc{%
F^{h}_{Cyc}(S\sma E\aF_{p+},\tilde{X})
\ar[r]^{\simeq_{p}} \ar[d]
&\ar[d]^{\simeq} F^{h}_{Cyc}(S\sma E\aF_{p+},X\phat) 
\\
F^{h}_{Cyc}(S_{TC;p}\sma E\aF_{p+},\tilde{X})
\ar[r]^{\simeq_{p}}
&F^{h}_{Cyc}(S_{TC;p}\sma E\aF_{p+},X\phat)
\\
\ar[u] F^{h}_{Cyc}(S_{TC;p},\tilde{X})
\ar[r]^{\simeq_{p}}
&F^{h}_{Cyc}(S_{TC;p},X\phat) \ar[u]_{\simeq}
\\
TC(\tilde{X};p)\ar[r]^{\simeq_{p}} \ar[u]
&TC(X\phat;p) \ar[u]_{\simeq}
}
\]
which (varying $X$) extends to a diagram of natural transformations from the
homotopy category of $p$-cyclotomic spectra to the stable category.
The maps marked ``$\simeq$'' are isomorphisms in the stable category
(q.v.~Theorem~\ref{thm:STCp} and the proof of Theorem~\ref{thm:weakrep})
and the maps marked ``$\simeq_{p}$'' become isomorphisms after
$p$-completion (q.v.~Lemma~\ref{lem:pcompF} and Theorem~\ref{thm:tcptc}).
The top left entry $F^{h}_{Cyc}(S\sma E\aF_{p+},\tilde{X})$ represents
the derived 
mapping spectrum $\bR F_{Cyc}(S,X)$ and the bottom left entry
represents the derived functor $TC(X;p)$.
\end{proof}

We have corresponding corepresentability results for cyclotomic spectra.

\begin{cons}\label{cons:STR}
Let $S_{TR}$ be the cyclotomic spectrum with underlying
orthogonal $\G$-spectrum 
\[
S_{TR}=\bigvee_{m\geq 0} F_{0}(\G/C_{m})_{+}
\]
and structure map $\cyc_{n}$ from 
\begin{multline*}
\rho^{*}_{n}\Phi^{C_{n}}S_{TR} \cong
\bigvee_{m \geq 0} \rho^{*}_{n}\Phi^{C_{n}}(F_{0}(\G/C_{m})_{+})
\cong \bigvee_{n\mid m, m \geq 0} \rho^{*}_{n}F_{0}(\G/C_{m})_{+}
\cong \bigvee_{m \geq 0} F_{0}(\G/C_{m})_{+}
\end{multline*}
to $S_{TR}$ induced by the canonical isomorphism.
\end{cons}

\begin{thm}\label{thm:TRpre}
The right derived functor of $TR(-)$ is corepresentable in the
homotopy category of cyclotomic spectra, with corepresenting object
$S_{TR}$. 
\end{thm}

\begin{proof}
For any pre-cyclotomic spectrum $X$, we have 
\begin{equation}\label{eq:TRpre}
F^{h}(S_{TR},X) \cong \lholim_{n\in \Theta} F^{\G}(\rho_{*}^{n}\Phi^{C_{n}}S_{TR},X)
\cong \lholim_{n\in \Theta} \prod_{m \geq 1} X^{C_{m}} 
\end{equation}
where the $x$ maps are the identity and the $y$ maps are induced by
the cyclotomic structure maps on $X$.  We show that this is level
equivalent to $\holim_{R}X^{C_{m}}$, which is $TR(X)$ when $X$ is
a fibrant cyclotomic spectrum.   As in the proof of
Theorem~\ref{thm:newF}, we 
use the Fubini theorem for homotopy limits combined with
Proposition~\ref{prop:Thetap} to compare the homotopy limit over
$\Theta_{p}$ with the homotopy equalizer of $x_{p},y_{p}$.

Let $P_{r}=\{p_{1},\dotsc,p_{r}\}$ denote the first $r$ prime numbers and let
$\Theta_{P_{r}}$ denote the full subcategory of $\Theta$ consisting of
the objects $n$ that only have elements of $P_{r}$ in their prime
factorization.  Then $\Theta_{P_{r}}$ is canonically isomorphic to
$\Theta_{p_{1}}\times \dotsb \times \Theta_{p_{r}}$.  Using the
Bousfield-Kan model for 
the homotopy limit, we have a canonical isomorphism 
\[
\lim\nolimits_{r}\holim_{\Theta_{p_{r}}}\dotsb \holim_{\Theta_{p_{1}}}F\iso
\lim\nolimits_{r}\holim_{\Theta_{P_{r}}}F\iso \holim_{\Theta}F
\]
for any functor $F$ from $\Theta$ to orthogonal spectra.  Likewise,
letting $R_{P_{r}}$ be the corresponding subcategory of $R$ maps, we
have a canonical isomorphism
\[
\lim\nolimits_{r}\holim_{R_{p_{r}}}\dotsb \holim_{R_{p_{1}}}F\iso
\lim\nolimits_{r}\holim_{R_{P_{r}}}F\iso \holim_{R}F.
\]
We note that all of the sequential limits above are limits of
towers of levelwise fibrations (since on the cosimplicial objects
at each cosimplicial level the map is a product projection).
Let
\[
F(p^{s_{1}},\dotsc,p^{s_{r}})=F^{\G}(\rho_{*}^{n}\Phi^{C_{n}}S_{TR},X),
\qquad n=p_{1}^{s_{1}}\dotsb p_{r}^{s_{r}}
\]
be the restriction of $F^{\G}(\rho_{*}^{n}\Phi^{C_{n}}S_{TR},X)$ to
$\Theta_{P_{r}}$.  Applying Proposition~\ref{prop:Thetap}, we have a
level equivalence 
\[
\holim_{\Theta_{p_{r}}}\dotsb \holim_{\Theta_{p_{1}}} F\to 
\hoequalizer_{x_{p_{r}},y_{p_{r}}}(\dotsb 
(\hoequalizer_{x_{p_{1}},y_{p_{1}}}F)\dotsb ).
\]
Applying~\eqref{eq:TRpre}, the homotopy equalizer over
$x_{p},y_{p}$ is the pullback of the diagram
\[
\xymatrix{%
&\prod_{m} (X^{C_{m}})^{I}\ar[d]\\
\prod_{m} X^{C_{m}}\ar[r]_-{(\id,R_{p})}
&\prod_{m} X^{C_{m}}\times \prod_{m} X^{C_{m}},
}
\]
which we can identify as the microscope of 
\[
s\mapsto \prod_{p\nmid m}X^{C_{mp^{s}}}
\]
over the maps $R_{p}$.  By
induction, we get compatible maps
\[
\holim_{\Theta_{p_{r}}}\dotsb
\holim_{\Theta_{p_{1}}}F(p^{s_{1}},\dotsc,p^{s_{r}})\to 
\prod\Mic_{s_{r}}\dotsb \Mic_{s_{1}} X^{C_{mn}},\qquad 
\]
where $n=p_{1}^{s_{1}}\dotsb p_{r}^{s_{r}}$
and the product is over $m$ not divisible by any element of $P_{r}$.
As we vary $r$, the tower on the right is again a tower of level
fibrations, and the limit is then isomorphic (via projection onto the
$m=1$ factor) to the limit
\[
\lim\nolimits_{r}\Mic_{s_{r}}\dotsb \Mic_{s_{1}} X^{C_{n}}.
\]
A functor out of $R_{p}$ is just a tower, so we have the canonical
level equivalence
\[
\holim_{R}X^{C_{n}}\iso
\lim\nolimits_{r}
\holim_{R_{p_{r}}}\dotsb \holim_{R_{p_{1}}} X^{C^{n}}\to 
\lim\nolimits_{r} \Mic_{s_{r}}\dotsb \Mic_{s_{1}}X^{C_{n}}.
\qedhere
\] 
\end{proof}

We next construct the representing pre-cyclotomic spectrum for $TC$.
Now instead of being a telescope, it will arise as the homotopy
colimit over the partially ordered set of positive
integers under divisibility, which we denote as $\bJ$.

\begin{cons}\label{cons:STC}
Let $S_{TC^{m}}$ (for $m\geq 1$) be the pre-cyclotomic spectrum whose
underlying orthogonal $\G$-spectrum is $F_{0}(\G/C_{m})_{+}$ and
whose structure map 
\[
\cyc_{n}\colon 
\rho_{*}^{n}\Phi^{C_{n}}S_{TC^{m}}=
F_{0}\rho_{*}^{n}(\G/C_{m})^{C_{n}}_{+}\to
F_{0}(\G/C_{m})_{+}=S_{TC^{m}}
\]
is either the trivial map $*\to S_{TC^{m}}$ if $n\nmid m$ or induced
by the quotient map 
\[
\rho_{*}^{n}(\G/C_{m})^{C_{n}}=\G/C_{m/n}\to \G/C_{m}
\]
when $n\mid m$.  For any $n\geq 1$, we have a map $S_{TC^{m}}\to S_{TC^{mn}}$
induced by the quotient $\G/C_{m}\to \G/C_{mn}$ that is 
a map of pre-cyclotomic spectra; let $S_{TC}$ be the homotopy colimit
of $S_{TC^{m}}$ over $\bJ$.
\end{cons}

We can identify $TC(-)$ in terms of maps of pre-cyclotomic
spectra out of $S_{TC}$.

\begin{thm}\label{thm:STC}
For pre-cyclotomic spectra $X$, there is a natural 
level equivalence
\[
\holim_{\bI}X^{C_{m}}\overto{\simeq}
F^{h}_{Cyc}(S_{TC},X).
\]
Thus, if $X$ is a fibrant cyclotomic spectrum, there is a natural
level equivalence
\[
TC(X)\overto{\simeq}F^{h}_{Cyc}(S_{TC},X).
\]
\end{thm}

\begin{proof}
Write $\bI(m)$ for the full subcategory of $\bI$ consisting of $m$ and
its divisors.  Looking at $F^{h}(S_{TC^{m}},X)$ and using the fact
that $\Phi^{C_{n}}S_{TC^{m}}=*$ for $n\nmid m$, we can identify
$F^{h}_{Cyc}(S_{TC^{m}},X)$ as the homotopy limit
\[
F^{h}_{Cyc}(S_{TC^{m}},X)=\holim_{\bI(m)} X^{C_{n}},
\]
with the usual interpretation of $F$ and $R$.
Since $S_{TC}=\hocolim_{\bJ}S_{TC^{m}}$, we have
\[
F^{h}_{Cyc}(S_{TC},X)\iso\holim_{m\in \bJ}\holim_{n\in \bI(m)} X^{C_{n}}.
\]
The natural map in the statement is then the map
\[
\holim_{\bI} X^{C_{n}}\iso \lim\nolimits_{m\in \bJ}\holim_{n\in \bI(m)} X^{C_{n}}
\to F^{h}_{Cyc}(S_{TC},X).
\]
Taking a cofinal sequence of $m$ in $\bJ$ and applying
\cite[XI.9.1]{BousfieldKan}, we see that the map is a level
equivalence. 
\end{proof}

The collapse maps $\G/C_{m}\to *$ induce a map of pre-cyclotomic
spectra from $S_{TC}$ to the cyclotomic spectrum $S$, which is clearly
a finite complete $\aF_{\fin}$-equivalence of the underlying
orthogonal $\G$-spectra.  The argument of Theorem~\ref{thm:weakrep}
now generalizes to prove the following theorem.

\begin{thm}\label{thm:weakrepglobal}
For $X$ a fibrant weak cyclotomic spectrum, $TC(X)$ is naturally
weakly equivalent to $F^{h}_{Cyc}(S\sma E\aF_{\fin+},X)$, which represents
the derived mapping spectrum $\bR F_{Cyc}(S,X)$. 
\end{thm}

\appendix

\section[The geometric fixed point functor]%
{A technical result on the geometric fixed point functor for
orthogonal \texorpdfstring{$G$}{G}-spectra}\label{app:geneqres}

Let $G$ be a compact Lie group and $U$ a $G$-universe, and consider
the category of orthogonal $G$-spectra modeled on $U$
\cite[\S II]{MM}.  Fix a closed normal subgroup $N\lhd G$.  
Proposition~4.7 in Chapter~V of~\cite{MM} says that for cofibrant
orthogonal $G$-spectra $X$ and $Y$, the canonical natural map (of orthogonal
$G/N$-spectra)
\[
\Phi^{N}X\sma \Phi^{N}Y\to \Phi^{N}(X\sma Y)
\]
is an isomorphism.  In this appendix, we generalize this statement to
the case when $X$ is not cofibrant.

\begin{lem}\label{lem:appb}
Let $X$ and $Y$ be orthogonal $G$-spectra and assume that $Y$ is
cofibrant.  The canonical natural map 
\[
\Phi^{N}X\sma \Phi^{N}Y\to \Phi^{N}(X\sma Y)
\]
is an isomorphism.
\end{lem}

The proof occupies the entirety of the appendix.

Since both sides commute with pushouts over Hurewicz cofibrations, the
statement reduces to the case when $Y=F_{A}B_{+}$ for $B$ a $G$-CW
complex (space) and $A<U$.  Since smashing with an unbased $G$-space
commutes appropriately with geometric fixed points~\cite[V.4.6,V.4.7]{MM}, 
\[
\Phi^{N}T\sma (B_{+})^{N}\iso \Phi^{N}(T\sma B_{+}),
\]
the statement reduces to the case when $B=*$, and $Y=F_{A}S^{0}$.

Following the notation of \cite[V.4.1]{MM}, we write $E$ for the
extension 
\[
1\to N\overto{\iota} G \overto{\epsilon} J\to 1,
\]
where $J = G/N$, and we write $\sJ_{E}$ for $(\sJ_{G})^{N}$.  Recall that the
geometric fixed point functor $\Phi^{N}X$ is formed from the
$\sJ_{E}$-space $\Fix^{N}X$ 
\[
\Fix^{N}X(V)=(X(V))^{N}
\] 
by enriched left Kan extension to $\sJ_{J}$ along the 
functor $\phi \colon \sJ_{E}\to \sJ_{J}$ sending $V$ to $V^{N}$.
The evident functor
\[
\oplus \colon \sJ_{E} \sma \sJ_{E} \to \sJ_{E}
\]
sending $(V,W)$ to $V\oplus W$ induces a smash product
of $\sJ_{E}$-spaces.  Since the diagram
\[
\xymatrix{%
\sJ_{E} \sma \sJ_{E}\ar[r]^-{\oplus} \ar[d]_{\phi \sma \phi} 
&\sJ_{E}\ar[d]^{\phi}\\
\sJ_{J} \sma \sJ_{J}\ar[r]_-{\oplus}
&\sJ_{J}
}
\]
commutes, enriched left Kan extension along $\phi$ takes the smash product of 
$\sJ_{E}$-spaces to the smash product of orthogonal $G/N$-spectra.
Thus, we are reduced to showing that the canonical natural map
\[
\Fix^{N}X \sma \Fix^{N}(F_{A}S^{0})\to \Fix^{N}(X\sma F_{A}S^{0})
\]
is an isomorphism.  

We now write formulas for $(\Fix^{N}X\sma \Fix^{N}(F_{A}S^{0}))(V)$ and
$\Fix^{N}(X\sma F_{A}S^{0})(V)$.  By definition $X\sma F_{A}S^{0}$ is
the enriched left Kan extension of the functor
$\sJ_{G}\sma \sJ_{G}\to \sJ_{G}$, the $V$th space of which, we can
write as the coequalizer 
\[
\xymatrix@R-1pc{%
\relax\displaystyle
\hskip-2em
\bigvee_{W,W',Z,Z'<U}
  \sJ_{G}(W'\oplus Z',V)\sma (\sJ(W,W')\sma \sJ_{G}(Z,Z'))
    \sma X(W)\sma F_{A}S^{0}(Z)
\hskip-2em \ar@<.5ex>[d] \ar@<-.5ex>[d]
\\
\relax\displaystyle 
\bigvee_{W,Z<U}
\sJ_{G}(W\oplus Z,V)\sma X(W)\sma F_{A}S^{0}(Z).
}
\]
From the universal property of $F_{A}S^{0}$, we have that
$F_{A}S^{0}(Z)=\sJ_{G}(A,Z)$, and so the coequalizer above simplifies to
the coequalizer
\[
\xymatrix@R-1pc{%
\relax\displaystyle
\bigvee_{W,W'<U}
  \sJ_{G}(W'\oplus A,V)\sma \sJ(W,W') \sma X(W)
\ar@<.5ex>[d] \ar@<-.5ex>[d]
\\
\relax\displaystyle 
\bigvee_{W<U}
\sJ_{G}(W\oplus A,V)\sma X(W).
}
\]
Let $a=\dim |A|$.  We have that $\sJ_{G}(W\oplus A,V)=*$ when $\dim
V<a$, and so the coequalizer is just the basepoint unless $\dim V\geq a$.  Let
$n=\dim V-a$.  Since in $\sJ_{G}$, every $W$ is isomorphic to
$\bR^{m}$ for some $m$, every map in $\sJ_{G}$ from $W\oplus A$ to $V$
factors through a map from $\bR^{n}\oplus A$ to $V$.  The coequalizer
above then reduces to the coequalizer
\[
\xymatrix@R-1pc{%
\relax\displaystyle
\bigvee_{W<U}
  \sJ_{G}(\bR^{n}\oplus A,V)\sma \sJ(W,\bR^{n}) \sma X(W)
\ar@<.5ex>[d] \ar@<-.5ex>[d]
\\
\relax\displaystyle 
\sJ_{G}(\bR^{n}\oplus A,V)\sma X(\bR^{n}).
}
\]
Using the same observation on $\bR^{n}$, we can identify this with
the orbit space
\[
\sJ_{G}(\bR^{n}\oplus A,V)\sma_{O(n)} X(\bR^{n}).
\]
An analogous argument for $(\Fix^{N}X\sma \Fix^{N}(F_{A}S^{0}))(V)$ yields an
identification as the orbit space
\[
\sJ_{E}(W \oplus A, V) \sma_{O(W)^{N}} X(W)^{N},
\]
where $V$ is isomorphic to $W\oplus A$ as an orthogonal
$N$-representation.  (In the case when no such decomposition exists,
$(\Fix^{N}X\sma \Fix^{N}(F_{A}S^{0}))(V)$ is just the basepoint.)
Thus, we must show that the map
\begin{equation}\label{eq:appmap}
\theta \colon \sJ_{E}(W \oplus A, V) \sma_{O(W)^{N}} X(W)^{N}\to
(\sJ_{G}(\bR^{n}\oplus A,V)\sma_{O(n)} X(\bR^{n}))^{N}.
\end{equation}
is a homeomorphism.   We prove this by constructing an explicit
inverse homeomorphism below.  First, it is useful to describe $\theta$
concretely in terms of elements.  For this,
choose and fix an orthonormal basis of $W$, which we can regard as a
(non-equivariant) isometric isomorphism $h\colon \bR^{n}\to W$.  For a
representative element $(f,x)$ on the left, with $f\in \sJ_{E}(W\oplus
A,V)$ and $x\in X(W)^{N}$, 
\[
\theta(f,x)=(f\circ (h\oplus \id_{A}),(h^{-1})_{*}x)
\]
where $(h^{-1})_{*}$ denotes the (non-equivariant) map
$X(W)\to X(\bR^{n})$ associated to $h^{-1}\colon W\to \bR^{n}$.
In the formula, 
\begin{gather*}
f\circ (h\oplus \id_{A})\in \sJ_{G}(\bR^{n}\oplus A,V),\\
(h^{-1})_{*}x\in X(\bR^{n}),
\end{gather*}
and by the abstract definition of the map, we must have that 
\[
(f\circ (h\oplus \id_{A}),(h^{-1})_{*}x)\in \sJ_{G}(\bR^{n}\oplus A,V)\sma_{O(n)} X(\bR^{n})
\]
is:
\begin{enumerate}
\item Independent of the choice of $h$,
\item Independent of the choice of representative $(f,x)$,
\item In the $N$-fixed point subspace $(\sJ_{G}(\bR^{n}\oplus A,V)\sma_{O(n)} X(\bR^{n}))^{N}$;
\end{enumerate}
however, it is trivial to check each of these facts explicitly in
terms of the elementwise formula for $\theta$, and doing this check
will (we hope) help give some insight into how the formula works.  For
(i), if we chose a different $h'\colon \bR^{n}\to W$, 
then we would have $h'=h\circ \B$ for some $\B$ in $O(n)$, and
\[
(f\circ (h'\oplus \id_{A}),(h^{\prime-1})_{*}x)=
(f\circ (h\oplus \id_{A})\circ (\B\oplus \id_{A}),\B^{-1}_{*}((h^{-1})_{*}x))
\]
represents the same element of $\sJ_{G}(\bR^{n}\oplus A,V)\sma_{O(n)}
X(\bR^{n})$ as $(f\circ (h\oplus \id_{A}),(h^{-1})_{*}x)$ since the right
$O(n)$ action on $\sJ_{G}(\bR^{n}\oplus A,V)$ is by right
multiplication by $\B\oplus \id_{A}$.
For (ii), any other representative of $(f,x)$ is $(f\circ
(\eta\oplus \id_{A}),\eta^{-1}_{*}x)$ for some $\eta\in O(W)^{N}$; 
\[
\theta(f\circ (\eta\oplus \id_{A}),\eta^{-1}_{*}x)=
(f\circ ((\eta\circ h)\oplus \id_{A}),((\eta\circ h)^{-1})_{*}x).
\]
Since $h'=\eta\circ h$ is a (non-equivariant) isometric isomorphism
$\bR^{n}\to W$, it follows by~(i) that this represents the same
element of $\sJ_{G}(\bR^{n}\oplus A,V)\sma_{O(n)} X(\bR^{n})$ as
$\theta(f,x)=(f\circ (h\oplus \id_{A}),(h^{-1})_{*}x)$. Finally for
(iii), first note that $\theta$ is $G$-equivariant; this follows from
the abstract definition of the map, but again, we check it explicitly:
The action of $G$ on $\sJ_{G}(\bR^{n}\oplus A,V)\sma_{O(n)}
X(\bR^{n})$ is diagonal, so
\begin{align*}
g\cdot\theta(f,x)&=(g\cdot(f\circ (h\oplus \id_{A})),g\cdot((h^{-1})_{*}x))\\
&=((g\cdot f)\circ ((g\cdot h)\oplus \id_{A}), ((g\cdot h)^{-1})_{*}(g\cdot x))\\
&=((g\cdot f)\circ (h\oplus \id_{A}),(h^{-1})_{*}(g\cdot x))
&\text{(by (i), taking }h'=g\cdot h\text{)}\\
&=\theta(g\cdot f,g\cdot x).
\end{align*}
In the second equality, the $G$-action distributes over composition
with the isometric isomorphism $h\oplus \id_{A}$ on 
$\sJ_{E}$ by inspection and over the action by the isometric isomorphism $h^{-1}$ 
on $X$ by the definition of orthogonal spectrum.
When $g\in N$, we have $g\cdot f=f$ since $f\in \sJ_{E}(W\oplus
A,V)=(\sJ_{G}(W\oplus A,V))^{N}$, and $g\cdot x=x$ since $x\in X(W)^{N}$.

We construct an inverse to the map $\theta$ of~\eqref{eq:appmap} as follows.  We note that
$\sJ_{G}(\bR^{n}\oplus A,V)$ is the $G$-equivariant space of (non-equivariant)
isometric isomorphisms from $\bR^{n}\oplus A$ to $V$ plus a disjoint
basepoint.  Let $\Psi$ denote the subset of $\sJ_{G}(\bR^{n}\oplus A,V)\sma
X(\bR^{n})$ which the quotient map sends into the $N$ fixed points
$(\sJ_{G}(\bR^{n}\oplus A,V)\sma_{O(n)} X(\bR^{n}))^{N}$.
Writing $(f,x)$ for a typical element of
$\sJ_{G}(\bR^{n}\oplus A,V)\sma X(\bR^{n})$, consider an element
$(f,x)\in \Psi$ where $x$ is not the basepoint.  Then for every
$\nu$ in $N$, 
\[
(\nu\cdot f,\nu\cdot x) = 
(f\circ (\B_{\nu}^{-1}\oplus \id_{A}),\B_{\nu*}x)
\]
for some $\B_{\nu}$ in $O(n)$; note that $\B_{\nu}$ is determined
uniquely since $O(n)$ acts freely on $\sJ_{G}(\bR^{n}\oplus A,V)$.  In
particular, the restriction of $f$ to $A$ must be $N$-equivariant and
the image $f(\bR^{n})$ of the restriction to $\bR^{n}$ must be stable
under the action of $N$.  Thus, $f$ induces an $N$-equivariant
isometric isomorphism from $f(\bR^{n})\oplus A$ to $V$.  In 
particular, in the case when $V$ does not contain a $N$-representation
isomorphic to $A$, $(\sJ_{G}(\bR^{n}\oplus A,V)\sma_{O(n)}
X(\bR^{n}))^{N}$ consists of just the basepoint, and the
map $\theta$ of~\eqref{eq:appmap} is a homeomorphism. 

Otherwise, we fix the $N$-equivariant isometric isomorphism $V\iso W\oplus A$,
and for each non-basepoint element $(f,x)$ in $\Psi$, we choose an
$N$-equivariant 
isometric isomorphism
\[
g_{(f,x)}\colon f(\bR^{n})\to W.
\]
Then $g^{-1}_{(f,x)}$ together with the restriction of $f$ to $A$
specify an element $\gamma_{(f,x)}$ in 
\[
(\sJ_{G}(W\oplus A,V))^{N}=\sJ_{E}(W\oplus A,V)
\]
and $g_{(f,x)}\circ f$ is an element of $\sJ_{G}(\bR^{n},W)$.  The
hypothesis that $(f,x)$ is in $\Psi$ then implies
that $(g_{(f,x)}\circ f)_{*}x\in X(W)$ is an $N$ fixed point:
\begin{multline*}
\nu\cdot ((g_{(f,x)}\circ f)_{*}x)
= (\nu\cdot (g_{(f,x)}\circ f))_{*}(\nu\cdot x)
= (\nu\cdot g_{(f,x)}\circ \nu \cdot f)_{*}(\nu\cdot x)\\
= (g_{(f,x)}\circ f\circ \B^{-1}_{\nu})_{*}(\B_{\nu*}x)
= (g_{(f,x)}\circ f)_{*}x.
\end{multline*}
We then get a basepoint preserving function $\psi$ on 
$\Psi$ that sends $(f,x)$ to 
\[
(\gamma_{(f,x)},(g_{(f,x)}\circ f)_{*}x) \in
\sJ_{E}(W \oplus A, V) \sma_{O(W)^{N}} X(W)^{N}.
\]
Because any two possible choices of $g_{(f,x)}$ are related by an
element of $O(W)^{N}$, it follows that $\psi$ is independent of the
choice of $g_{(f,x)}$.  It is easy to see that $\psi$ is continuous at
the basepoint and since for any non-basepoint $x$, we can choose
$g_{(f,x)}$ locally to be a continuous function of $(f,x)$, it follows
that $\psi$ is continuous.  Finally, for $\B \in O(n)$, 
\begin{multline*}
\psi(f\circ (\B^{-1}\oplus \id_{A}),\B_{*}x)=
(\gamma_{(f,x)},(g_{(f,x)}\circ f\circ \B^{-1})_{*}(\B_{*}x))\\
=(\gamma_{(f,x)},(g_{(f,x)}\circ f)_{*}x)
=\psi(f,x),
\end{multline*}
and so $\psi$ descends to a continuous map
\[
(\sJ_{G}(\bR^{n}\oplus A,V)\sma_{O(n)} X(\bR^{n}))^{N}\to
\sJ_{E}(W \oplus A, V) \sma_{O(W)^{N}} X(W)^{N}.
\]
Using the formula above, both composites of $\theta$ and
$\psi$ are easily seen to be 
the appropriate identity maps.



\begin{thebibliography}{10}

\bibitem{BM4}
Andrew~J. Blumberg and Michael~A. Mandell.
\newblock Localization theorems in topological {H}ochschild homology and
  topological cyclic homology.
\newblock {\em Geom. Topol.}, 16(2):1053--1120, 2012.

\bibitem{BousfieldKan}
A.~K. Bousfield and D.~M. Kan.
\newblock {\em Homotopy limits, completions and localizations}.
\newblock Lecture Notes in Mathematics, Vol. 304. Springer-Verlag, Berlin,
  1972.

\bibitem{HM2}
Lars Hesselholt and Ib~Madsen.
\newblock On the {$K$}-theory of finite algebras over {W}itt vectors of perfect
  fields.
\newblock {\em Topology}, 36(1):29--101, 1997.

\bibitem{HMAnnals}
Lars Hesselholt and Ib~Madsen.
\newblock On the {$K$}-theory of local fields.
\newblock {\em Ann. of Math. (2)}, 158(1):1--113, 2003.

\bibitem{Hirschhorn}
Philip~S. Hirschhorn.
\newblock {\em Model categories and their localizations}, volume~99 of {\em
  Mathematical Surveys and Monographs}.
\newblock American Mathematical Society, Providence, RI, 2003.

\bibitem{Hovey-ModelCat}
Mark Hovey.
\newblock {\em Model categories}, volume~63 of {\em Mathematical Surveys and
  Monographs}.
\newblock American Mathematical Society, Providence, RI, 1999.

\bibitem{Kaledin-ICM2010}
D.~Kaledin.
\newblock Motivic structures in non-commutative geometry.
\newblock In {\em Proceedings of the {I}nternational {C}ongress of
  {M}athematicians. {V}olume {II}}, pages 461--496, New Delhi, 2010. Hindustan
  Book Agency.

\bibitem{MM}
M.~A. Mandell and J.~P. May.
\newblock Equivariant orthogonal spectra and {$S$}-modules.
\newblock {\em Mem. Amer. Math. Soc.}, 159(755):x+108, 2002.

\bibitem{MMSS}
M.~A. Mandell, J.~P. May, S.~Schwede, and B.~Shipley.
\newblock Model categories of diagram spectra.
\newblock {\em Proc. London Math. Soc. (3)}, 82(2):441--512, 2001.

\bibitem{May-Alaska}
J.~P. May.
\newblock {\em Equivariant homotopy and cohomology theory}, volume~91 of {\em
  CBMS Regional Conference Series in Mathematics}.
\newblock Published for the Conference Board of the Mathematical Sciences,
  Washington, DC, 1996.
\newblock With contributions by M. Cole, G. Comeza{\~n}a, S. Costenoble, A. D.
  Elmendorf, J. P. C. Greenlees, L. G. Lewis, Jr., R. J. Piacenza, G.
  Triantafillou, and S. Waner.

\bibitem{MayPonto}
J.~P. May and K.~Ponto.
\newblock {\em More concise algebraic topology: localization, completion,
and model categories.}
\newblock Chicago Lectures in Mathematics.  Chicago University Press,
Chicago, 2012.

\bibitem{Quillen-HA}
Daniel~G. Quillen.
\newblock {\em Homotopical algebra}.
\newblock Lecture Notes in Mathematics, No. 43. Springer-Verlag, Berlin, 1967.

\bibitem{ABC}
A.~Radulescu-Banu.
\newblock Cofibrations in homotopy theory.
\newblock Preprint arXiv:math/0610009v4, 2009.

\end{thebibliography}
\end{document}